\documentclass[12pt,reqno]{amsart}
\usepackage{amsthm, amscd, amsfonts, amssymb, graphicx, color}
\usepackage[bookmarksnumbered, colorlinks, plainpages,linkcolor=green,anchorcolor=blue,citecolor=blue,urlcolor=blue]{hyperref}
\usepackage{enumerate}
\usepackage{exscale}
\usepackage[all]{xy}
\usepackage{relsize}
\usepackage{pdfsync}
\numberwithin{equation}{section}

\makeatletter
\@namedef{subjclassname@2020}{%
  \textup{2020} Mathematics Subject Classification}
\makeatother

\usepackage[includemp,body={398pt,550pt},footskip=30pt,%
            marginparwidth=60pt,marginparsep=10pt]{geometry}

\newtheorem{theorem}{Theorem}[section]
\newtheorem{lemma}[theorem]{Lemma}
\newtheorem{proposition}[theorem]{Proposition}
\newtheorem{corollary}[theorem]{Corollary}
\theoremstyle{definition}
\newtheorem{definition}[theorem]{Definition}

\numberwithin{equation}{section}

\newtheorem{theory}{Theorem}

\newtheorem{lem}[theory]{Lemma}

\newcommand{\D}{\mathbb{D}}
\newcommand{\h}{\mathcal{H}}

\newcommand{\C}{\mathbb{C}}

\newcommand{\Z}{\mathbb{Z}}
\newcommand{\res}{\mathrm{restricted}}

\begin{document}
\title[Absolutely summing Carleson embeddings]{Absolutely summing Carleson embeddings on weighted Fock spaces with $A_{\infty}$-type weights}
\date{April 1, 2024. Revised in June 7, 2024.
}
\author[J. Chen, B. He and M. Wang]{Jiale Chen, Bo He and Maofa Wang}
\address{Jiale Chen, School of Mathematics and Statistics, Shaanxi Normal University, Xi'an 710119, China.}
\email{jialechen@snnu.edu.cn}
\address{Bo He, Department of Mathematics, Shantou University, Shantou 515063, China.}
\email{16bhe@stu.edu.cn}
\address{Maofa Wang, School of Mathematics and Statistics, Wuhan University, Wuhan 430072, China.}
\email{mfwang.math@whu.edu.cn}

\thanks{Chen was supported by the Fundamental Research Funds for the Central Universities (No. GK202207018) of China. He was supported by the NNSF (No. 12071272) of China, NSF of Guangdong Province (Nos. 2022A1515012117, 2024A1515010468) and LKSF STU-GTIIT Joint-research Grant (No. 2024LKSFG06). Wang was supported by the NNSF (No. 12171373) of China.}



\begin{abstract}
  \noindent In this paper, we investigate the $r$-summing Carleson embeddings on weighted Fock spaces $F^p_{\alpha,w}$. By using duality arguments, translating techniques and block diagonal operator skills, we completely characterize the $r$-summability of the natural embeddings $I_d:F^p_{\alpha,w}\to L^p_{\alpha}(\mu)$ for any $r\geq1$ and $p>1$, where $w$ is a weight on the complex plane $\C$ that satisfies an $A_p$-type condition. As applications, we establish some results on the $r$-summability of differentiation and integration operators, Volterra-type operators and composition operators. Especially, we completely characterize the boundedness of Volterra-type operators and composition operators on vector-valued Fock spaces for all $1<p<\infty$, which were left open before for the case $1<p<2$.
\end{abstract}
\maketitle


\section{Introduction}
\allowdisplaybreaks[4]
Recently, absolute summability of operators acting on various spaces of analytic functions has drawn many attentions \cite{BGJJ,Do99,FL20,FL22,HJLL,JL,LR18}. In particular, Lef\`{e}vre and Rodr\'{i}guez-Piazza \cite{LR18} characterized the $r$-summing Carleson embeddings on the Hardy spaces $H^p(\D)$ for $p>1$ and $r\geq1$. Later on, He et al. \cite{HJLL} established the corresponding result for the Bergman spaces $A^p(\D)$. In this paper, we are going to investigate the $r$-summing Carleson embeddings on weighted Fock spaces.

We now introduce the spaces to work on. A nonnegative function $w$ on the complex plane $\C$ is said to be a weight if it is locally integrable on $\C$. Given $\alpha,p>0$ and a weight $w$ on $\C$, the weighted space $L^p_{\alpha,w}$ consists of measurable functions $f$ on $\C$ such that
$$\|f\|^p_{L^p_{\alpha,w}}:=\int_{\C}|f(z)|^pe^{-\frac{p\alpha}{2}|z|^2}w(z)dA(z)<\infty,$$
where $dA$ is the Lebesgue measure on $\C$. The weighted Fock space $F^p_{\alpha,w}$ is defined by
$$F^p_{\alpha,w}:=L^p_{\alpha,w}\cap\mathcal{H}(\C)$$
with inherited (quasi-)norm, where $\mathcal{H}(\C)$ is the set of entire functions. If $w\equiv1$, then we obtain the standard Fock spaces $F^p_{\alpha}$. We refer to \cite{Zh} for a brief account on Fock spaces.

For $t>0$ and $z\in\C$, we write $Q_t(z)$ for the square in $\C$ centered at $z$ with side length $t$, whose sides are parallel to the coordinate axes. As usual, for $1\leq p\leq\infty$, $p'$ denotes the conjugate exponent of $p$, i.e. $1/p+1/p'=1$. Given $1<p<\infty$, we say a weight $w$ belongs to the class $A^{\res}_p$ if $w(z)>0$ for almost every $z\in\C$ and there exists some $t>0$ such that
\begin{equation}\label{Cpr}
\mathcal{C}_{p,t}(w):=\sup_{z\in\C}\left(\frac{1}{A(Q_t(z))}\int_{Q_t(z)}wdA\right)
\left(\frac{1}{A(Q_t(z))}\int_{Q_t(z)}w^{-\frac{p'}{p}}dA\right)^{\frac{p}{p'}}<\infty.
\end{equation}
The class $A^{\res}_p$ was introduced by Isralowitz \cite{Is}, and it was shown that the condition \eqref{Cpr} is independent of the choice of $t$: if $\mathcal{C}_{p,t_0}(w)<\infty$ for some $t_0>0$, then $\mathcal{C}_{p,t}(w)<\infty$ for any $t>0$. The purpose of introducing $A^{\res}_p$ is to characterize the boundedness of the Fock projections. Recall that for $\alpha>0$, the Fock projection $P_{\alpha}$ is defined for $f\in L^1_{\mathrm{loc}}(\C,dA)$ by
$$P_{\alpha}(f)(z):=\frac{\alpha}{\pi}\int_{\C}f(u)e^{\alpha\bar{u}z}e^{-\alpha|u|^2}dA(u),\quad z\in\C.$$
Isralowitz \cite{Is} proved that for $p>1$, $P_{\alpha}$ is bounded on $L^p_{\alpha,w}$ if and only if $w\in A^{\mathrm{restricted}}_p$. Later on, Cascante, F\`{a}brega and Pel\'{a}ez \cite{CFP} introduced the class $A^{\res}_1$, consisting of weights $w$ such that $w(z)>0$ a.e. on $\C$ and for some fixed $t>0$,
$$\mathcal{C}_{1,t}(w):=\sup_{z\in\C}\frac{\int_{Q_t(z)}wdA}{A(Q_t(z))\mathrm{ess}\inf_{u\in Q_t(z)}w(u)}<\infty,$$
and proved that $P_{\alpha}$ is bounded on $L^1_{\alpha,w}$ if and only if $w\in A^{\res}_1$. Similarly to the Muckenhoupt weights, we define
$$A^{\res}_{\infty}:=\bigcup_{1\leq p<\infty}A^{\res}_p.$$
In \cite{CFP}, Cascante, F\`{a}brega and Pel\'{a}ez established some properties of the weighted Fock spaces $F^p_{\alpha,w}$ induced by $w\in A^{\res}_{\infty}$. In particular, the boundedness of Carleson embeddings was characterized.

For a positive Borel measure $\mu$ on $\C$, let $L^p_{\alpha}(\mu)$ be the space of measurable functions $f$ on $\C$ such that
$$\|f\|^p_{L^p_{\alpha}(\mu)}:=\int_{\C}|f(z)|^pe^{-\frac{p\alpha}{2}|z|^2}d\mu(z)<\infty.$$
Write $D(z,t)$ for the disk centered at $z\in\C$ with radius $t>0$, and write $w(E):=\int_EwdA$ for measurable subsets $E\subset\C$. It was shown in \cite{CFP} that the natural embedding $I_d:F^p_{\alpha,w}\to L^p_{\alpha}(\mu)$ is bounded if and only if
$$\widehat{\mu}_{w}(z):=\frac{\mu(D(z,1))}{w(D(z,1))}$$
is bounded on $\C$. Moreover, using the same method, one can easily obtain that $I_d:F^p_{\alpha,w}\to L^p_{\alpha}(\mu)$ is compact if and only if $\lim_{|z|\to\infty}\widehat{\mu}_{w}(z)=0$. The purpose of this paper is to characterize the $r$-summability of the embedding $I_d:F^p_{\alpha,w}\to L^p_{\alpha}(\mu)$.

\begin{definition}
Given $r\geq1$, a linear operator $T:X\to Y$ between Banach spaces $X$ and $Y$ is said to be $r$-summing if there exists $C\geq0$ such that
$$\left(\sum_{j=1}^n\|Tx_j\|^r_Y\right)^{1/r}\leq
C\sup_{x^*\in B_{X^*}}\left(\sum_{j=1}^n|x^*(x_j)|^r\right)^{1/r}
=C\sup_{\{c_j\}\in B_{l^{r'}}}\left\|\sum_{j=1}^nc_jx_j\right\|_X$$
for every finite sequence $\{x_j\}_{1\leq j\leq n}\subset X$. The $r$-summing norm of $T$, denoted by $\pi_r(T)$, is the least suitable constant $C$.
\end{definition}

We write $\Pi_r(X,Y)$ to denote the set of all $r$-summing operators mapping $X$ into $Y$. The operators in $\Pi_1(X,Y)$ are also called to be absolutely summing. The roots of the theory of $r$-summing operators can be traced back to the work of Grothendieck \cite{Gr53}. However, it was not until ten years later that Pietsch \cite{Pi66} clearly defined this class of operators and established many fundamental properties of them. After that, the theory of $r$-summing operators developed quickly; see \cite{DJT,Pe,Wo} for more information.

Our main result is as follows, which characterizes the $r$-summability of $I_d:F^p_{\alpha,w}\to L^p_{\alpha}(\mu)$ for $p>1$ and $r\geq1$. We write $L^p:=L^p(\C,dA)$ to save the notation.

\begin{theorem}\label{main}
Let $\alpha>0$ and $\mu$ be a positive Borel measure on $\C$.
\begin{enumerate}[(1)]
  \item Let $1<p<2$ and $w\in A^{\res}_p$. Then for any $r\geq1$, $I_d:F^p_{\alpha,w}\to L^p_{\alpha}(\mu)$ is $r$-summing if and only if $\widehat{\mu}_w\in L^{2/p}$. Moreover,
      $$\pi_r(I_d)\asymp\pi_2(I_d)\asymp\left(\int_{\C}\widehat{\mu}_w(z)^{2/p}dA(z)\right)^{1/2}.$$
  \item Let $p\geq2$ and $w\in A^{\res}_{\infty}$. Then
  \begin{enumerate}[(i)]
    \item for $1\leq r\leq p'$, $I_d:F^p_{\alpha,w}\to L^p_{\alpha}(\mu)$ is $r$-summing if and only if $\widehat{\mu}_w\in L^{p'/p}$. Moreover,
        $$\pi_r(I_d)\asymp\pi_1(I_d)\asymp\left(\int_{\C}\widehat{\mu}_w(z)^{p'/p}dA(z)\right)^{1/p'}.$$
    \item for $p'\leq r\leq p$, $I_d:F^p_{\alpha,w}\to L^p_{\alpha}(\mu)$ is $r$-summing if and only if $\widehat{\mu}_w\in L^{r/p}$. Moreover,
        $$\pi_r(I_d)\asymp\left(\int_{\C}\widehat{\mu}_w(z)^{r/p}dA(z)\right)^{1/r}.$$
    \item for $r\geq p$, $I_d:F^p_{\alpha,w}\to L^p_{\alpha}(\mu)$ is $r$-summing if and only if $\widehat{\mu}_w\in L^1$. Moreover,
        $$\pi_r(I_d)\asymp\pi_p(I_d)\asymp\left(\int_{\C}\widehat{\mu}_w(z)dA(z)\right)^{1/p}.$$
  \end{enumerate}
\end{enumerate}
Here, $\widehat{\mu}_{w}(z)=\frac{\mu(D(z,1))}{w(D(z,1))}$ for $z\in\C$.
\end{theorem}

Generally speaking, there exist compact operators between Banach spaces which fail to be $r$-summing for all $1\leq r<\infty$, and there exist $1$-summing operators between Banach spaces which are not compact (see \cite[p. 38]{DJT}). However, it follows from Theorem \ref{main} that, for $1<p<\infty$ and $w\in A^{\res}_p$, if $I_d:F^p_{\alpha,w}\to L^p_{\alpha}(\mu)$ is $r$-summing for some $r\geq1$, then it is compact.

Our approach to Theorem \ref{main} is inspired by \cite{HJLL,LR18}. However, there are some new difficulties. The first one is that the Gaussian weight $e^{-\frac{p\alpha}{2}|z|^2}$ in the definitions of $F^p_{\alpha,w}$ and $L^p_{\alpha}(\mu)$ depends on the choice of $p$. The second one is that the weights $w\in A^{\res}_{\infty}$ are usually non-radial, and for the spaces $F^2_{\alpha,w}$, the reproducing kernels do not have explicit forms. Hence some new techniques are needed. In the case $1<p<2$, we first use a representation of the dual space of $F^p_{\alpha,w}$ to characterize the $2$-summability of $I_d:F^p_{\alpha,w}\to L^2_{\alpha}(\mu)$. Then we translate the general case $I_d:F^p_{\alpha,w}\to L^q_{\alpha}(\mu)$ ($1\leq q<2$) into the special case $q=2$, and connect the $2$-summability of $I_d:F^p_{\alpha,w}\to L^q_{\alpha}(\mu)$ with the boundedness of the Berezin transform $B_{\alpha}$ between some weighted Lebesgue spaces, which allows us to determine the $r$-summing embeddings $I_d:F^p_{\alpha,w}\to L^q_{\alpha}(\mu)$ for all $1<p<2$, $1\leq q\leq2$ and $r\geq1$ (see Theorem \ref{<}). In the case $p\geq2$, we characterize the $r$-summing embeddings $I_d:F^p_{\alpha,w}\to L^p_{\alpha}(\mu)$ via $r$-summing multipliers on the spaces $l^p$ and some block diagonal operators between the $l^p$-sums of sequences of some ``local'' spaces.

As an immediate application of Theorem \ref{main}, we can characterize the $r$-summing composition operators on weighted Fock spaces. Given $\varphi\in\h(\C)$, the composition operator $C_{\varphi}$ is defined by $C_{\varphi}f:=f\circ\varphi$, $f\in\h(\C)$. The boundedness and compactness of composition operators acting on $F^p_{\alpha}$ were determined in \cite{CMS}. Recently, Chen \cite{Ch23} investigated some properties, including the boundedness and compactness, of composition operators on weighted Fock spaces $F^p_{\alpha,w}$ induced by $w\in A^{\res}_{\infty}$. Based on Theorem \ref{main} and the following formula:
$$\|C_{\varphi}f\|^p_{F^p_{\alpha,w}}=\int_{\C}|f(\varphi(z))|^pe^{-\frac{p\alpha}{2}|z|^2}w(z)dA(z)
=\int_{\C}|f(z)|^pd\mu_{\varphi}^{p,\alpha,w}(z),$$
where the pull-back measure $\mu_{\varphi}^{p,\alpha,w}$ is defined for Borel sets $E\subset\C$ by $\mu^{p,\alpha,w}_{\varphi}(E):=\int_{\varphi^{-1}(E)}e^{-\frac{p\alpha}{2}|z|^2}w(z)dA(z)$, one can easily obtain the characterizations for the $r$-summability of composition operators acting on $F^p_{\alpha,w}$. However, in the standard
Fock space setting, we can obtain a more explicit result: for $p>1$ and $r\geq1$, $C_{\varphi}$ is $r$-summing on $F^p_{\alpha}$ if and only if $\varphi(z)=az+b$ for some $a,b\in\C$ with $|a|<1$. A striking consequence of this fact is that we can completely determine the boundedness of composition operators between some vector-valued Fock spaces. Given $p\geq1$, $\alpha>0$ and a complex Banach space $X$, the $X$-valued Fock space $F^p_{\alpha}(X)$ is defined to be the space of $X$-valued entire functions $f$ such that
$$\|f\|^p_{F^p_{\alpha}(X)}:=\int_{\C}\|f(z)\|^p_Xe^{-\frac{p\alpha}{2}|z|^2}dA(z)<\infty.$$
We also define the weak space $wF^p_{\alpha}(X)$, consisting of $X$-valued entire functions $f$ with
$$\|f\|^p_{wF^p_{\alpha}(X)}:=\sup_{x^*\in B_{X^*}}\|x^*\circ f\|_{F^p_{\alpha}}<\infty.$$
It is clear that $F^p_{\alpha}(X)\subset wF^p_{\alpha}(X)$. The boundedness of composition operators acting from weak to strong spaces of vector-valued analytic functions was first investigated in \cite{LTW} in the setting of Bergman and Hardy spaces. Recently, Chen and Wang \cite[Corollary 4.7]{CW21} shown that for $p\geq2$, $\alpha>0$ and any complex infinite-dimensional Banach space $X$, $C_{\varphi}:wF^p_{\alpha}(X)\to F^p_{\alpha}(X)$ is bounded if and only if $\varphi(z)=az+b$ for some $a,b\in\C$ with $|a|<1$. However, the method used in \cite{CW21,LTW} does not work for the case $1\leq p<2$. We can fortunately apply the characterization of $r$-summing composition operators on $F^p_{\alpha}$ to obtain the boundedness of $C_{\varphi}:wF^p_{\alpha}(X)\to F^p_{\alpha}(X)$ for the case $1<p<2$. Correspondingly, we also apply Theorem \ref{main} to characterize the $r$-summing Volterra-type operators on $F^p_{\alpha}$, and the bounded Volterra-type operators from weak spaces $wF^p_{\alpha}(X)$ to strong spaces $F^p_{\alpha}(X)$. 

The rest part of this paper is organized as follows. In Section \ref{pre}, we recall some preliminary results. Section \ref{proof} is devoted to the proof of Theorem \ref{main}. In Section \ref{app}, we apply Theorem \ref{main} to establish some results on the $r$-summability of differentiation and integration operators, Volterra-type operators and composition operators. Furthermore, we characterize the boundedness of Volterra-type and composition operators from weak to strong vector-valued Fock spaces, which complements the corresponding results in \cite{CW21}.

Throughout the paper, the notation $A\lesssim B$ (or $B\gtrsim A$) means that there exists a nonessential constant $C>0$ such that $A\leq CB$. If $A\lesssim B\lesssim A$, then we write $A\asymp B$. For a Banach space $X$, we use $B_X$ to denote the open unit ball of $X$, and if $X=L^p(\Omega,\Sigma,m)$ for some $p\geq1$ and some measure space $(\Omega,\Sigma,m)$, we write $B^+_X$ for the set of positive elements in $B_X$. For a subset $E\subset\C$, $\chi_E$ denotes its characteristic function.

\section{Preliminaries}\label{pre}

In this section, we give some preliminary results that will be used in the sequel.

\subsection{Absolutely summing operators}

We first recall some basic facts on $r$-summing operators, which can be found in \cite{DJT}.

It is well-known that for any $r\geq1$, the class $\Pi_r(X,Y)$ forms an operator ideal between Banach spaces: for any $T\in\Pi_r(X,Y)$, and for any two Banach spaces $X_0,Y_0$ such that both $S:X_0\to X$ and $U:Y\to Y_0$ are linear bounded operators, we have $UTS\in\Pi_r(X_0,Y_0)$ with
$$\pi_r(UTS)\leq\|U\|\pi_r(T)\|S\|.$$

We next recall that the class $\Pi_r(X,Y)$ is monotone with respect to $r$, that is, if $1\leq r_1\leq r_2<\infty$, then we have $\Pi_{r_1}(X,Y)\subset\Pi_{r_2}(X,Y)$, and for any $T\in\Pi_{r_1}(X,Y)$, $\pi_{r_2}(T)\leq\pi_{r_1}(T)$.

The following theorem gives an equivalent description of $r$-summing operators, which can be found in \cite[p. 56]{DJT}.

\begin{theory}\label{sum-T}
{\it
Let $X,Y$ be Banach spaces and $r\geq1$. Then a bounded operator $T:X\to Y$ is $r$-summing if and only if there exists $C\geq0$ such that for any measure space $(\Omega,\Sigma,m)$ and any measurable function $F:\Omega\to X$, one has
$$\int_{\Omega}\|T\circ F\|_Y^rdm\leq C\sup_{\xi\in B_{X^*}}\int_{\Omega}|\xi\circ F|^rdm.$$
Moreover, the best admissible $C$ is $\pi_r(T)^r$.
}
\end{theory}

The following lemma reveals the relation between order bounded operators and $r$-summing operators, which is a direct consequence of \cite[Propositions 5.5 and 5.18]{DJT}. Recall that for a Banach space $X$ and a measure space $(\Omega,\Sigma,m)$, an operator $T:X\to L^p(\Omega,m)$ is said to be order bounded if there is a non-negative function $h\in L^p(\Omega,m)$ such that $|Tf|\leq h$ $m$-almost everywhere for each $f\in B_X$.

\begin{lem}\label{order}{\it
Let $X$ be a Banach space, $p\geq1$, and let $(\Omega,\Sigma,m)$ be a measure space. If the operator $T:X\to L^p(\Omega,m)$ is order bounded, then it is $p$-summing with
$$\pi_p(T)\leq\left\|\sup_{f\in B_X}|Tf|\right\|_{L^p(\Omega,m)}.$$}
\end{lem}

We will also need a characterization for the $r$-summability of multiplication operators acting on sequence spaces. Given a sequence $\lambda=\{\lambda_{\nu}\}_{\nu\in\Z^2}$, the multiplication operator $M_{\lambda}$ is defined by
$$M_{\lambda}c:=\{\lambda_{\nu}c_{\nu}\}_{\nu\in\Z^2},\quad c=\{c_{\nu}\}_{\nu\in\Z^2}.$$
The following lemma gives the $r$-summing norm estimates for $M_{\lambda}$ acting on $l^p(\Z^2)$, which can be found in \cite[Proposition 4.1]{LR18}.

\begin{lem}\label{mul}{\it
Let $p\geq2$ and $r\geq1$. Consider the multiplication operator $M_{\lambda}:l^p(\Z^2)\to l^p(\Z^2)$. Then
\begin{enumerate}[(1)]
  \item for $r\leq p'$, $\pi_r(M_{\lambda})\asymp\|\lambda\|_{l^{p'}(\Z^2)}$.
  \item for $p'\leq r\leq p$, $\pi_r(M_{\lambda})\asymp\|\lambda\|_{l^r(\Z^2)}$.
  \item for $r\geq p$, $\pi_r(M_{\lambda})\asymp\|\lambda\|_{l^p(\Z^2)}$.
\end{enumerate}}
\end{lem}

\subsection{Estimates involving weights}

We now recall some estimates involving weights in $A^{\res}_{\infty}$.

Note first that if $w\in A^{\res}_p$ for some $p>1$, then $w':=w^{-p'/p}\in A^{\res}_{p'}$ and
\begin{equation}\label{w'}
w(Q_1(z))\asymp w'(Q_1(z))^{-p/p'},\quad z\in\C.
\end{equation}
Consequently, we also have that $\hat{w}(z):=\int_{Q_1(z)}wdA$ belongs to $A^{\res}_p$.

We will need the following estimates, which can be found in \cite[Remark 2.3]{CFP} and \cite[Lemma 3.4]{Is}. We here treat $\mathbb{Z}^2$ as a subset of $\C$ in the natural way.

\begin{lem}\label{esti}{\it
Let $w\in A^{\mathrm{restricted}}_{\infty}$.
\begin{enumerate}[(1)]
  \item There exists $C>0$ such that for any $\nu,\nu'\in\mathbb{Z}^2$,
  $$\frac{w(Q_1(\nu))}{w(Q_1(\nu'))}\leq C^{|\nu-\nu'|}.$$
  \item For any $z,u\in\C$ with $|z-u|<3$ and any $1/4\leq t_1,t_2\leq1$,
  $$w(Q_{t_1}(z))\asymp w(Q_{t_1}(u))\asymp w(D(z,t_1))\asymp w(D(u,t_1))\asymp w(D(u,t_2)).$$
\end{enumerate}}
\end{lem}

The following lemma establishes some pointwise estimates for entire functions, which was proved in \cite[Lemma 3.1]{CFP}.

\begin{lem}\label{pointwise}{\it
Let $\alpha,p,t>0$ and $w\in A^{\res}_{\infty}$. Then for $f\in\h(\C)$ and $z\in\C$,
$$|f(z)|^pe^{-\frac{p\alpha}{2}|z|^2}\lesssim
\frac{1}{w(D(z,t))}\int_{D(z,t)}|f(u)|^pe^{-\frac{p\alpha}{2}|u|^2}w(u)dA(u),$$
where the implicit constant is independent of $f$ and $z$.
}
\end{lem}

For $u\in\C$, we write $K_u(z):=e^{\alpha\bar{u}z}$ for the reproducing kernels of the standard Fock space $F^2_{\alpha}$. The following lemma gives some test functions in $F^p_{\alpha,w}$; see \cite[Propositions 4.1 and 4.2]{CFP}. 

\begin{lem}\label{test}{\it
Let $\alpha,p>0$ and $w\in A^{\res}_{\infty}$.
\begin{enumerate}[(1)]
  \item For any $u\in\C$,
  $$\|K_u\|^p_{F^p_{\alpha,w}}\asymp e^{\frac{p\alpha}{2}|u|^2}w(D(u,1)).$$
  \item For any sequence $c=\{c_{\nu}\}\in l^p(\Z^2)$, the function
  $$f:=\sum_{\nu\in\Z^2}c_{\nu}\frac{K_{\nu}}{\|K_{\nu}\|_{F^p_{\alpha,w}}}$$
  belongs to $F^p_{\alpha,w}$ with $\|f\|_{F^p_{\alpha,w}}\lesssim\|c\|_{l^p(\Z^2)}$.
\end{enumerate}
}
\end{lem}

We end this section by the following estimate.

\begin{lemma}\label{dis}
Let $\gamma,\eta>0$, $w\in A^{\res}_{\infty}$, and let $\mu$ be a positive Borel measure on $\C$. Then
$$\int_{\C}\frac{\mu(D(z,1))^{\gamma}}{w(D(z,1))^{\eta}}dA(z)\asymp
\sum_{\nu\in\Z^2}\frac{\mu(Q_1(\nu))^{\gamma}}{w(Q_1(\nu))^{\eta}}.$$
\begin{proof}
Note that for $\nu\in\Z^2$, $Q_1(\nu)\subset D(z,1)$ for any $z\in Q_{1/3}(\nu)$. Therefore, Lemma \ref{esti} implies that
\begin{align*}
\sum_{\nu\in\Z^2}\frac{\mu(Q_1(\nu))^{\gamma}}{w(Q_1(\nu))^{\eta}}
\lesssim\sum_{\nu\in\Z^2}\int_{Q_{1/3}(\nu)}\frac{\mu(D(z,1))^{\gamma}}{w(D(z,1))^{\eta}}dA(z)
\leq\int_{\C}\frac{\mu(D(z,1))^{\gamma}}{w(D(z,1))^{\eta}}dA(z).
\end{align*}
To establish the converse estimate, write
$$\Gamma_z:=\{\nu\in\Z^2:D(z,1)\cap Q_1(\nu)\neq\emptyset\}$$
for $z\in\C$. Then it is easy to see $|\Gamma_z|\leq36$ for any $z\in\C$. In fact, choose mutually distinct elements $\nu_1,\dots,\nu_n\in\Gamma_z$, and choose $u_j\in D(z,1)\cap Q_1(\nu_j)$, $1\leq j\leq n$. Then for $u\in D(\nu_j,1/2)$,
$$|u-z|\leq|u-\nu_j|+|\nu_j-u_j|+|u_j-z|<\frac{1}{2}+\frac{\sqrt{2}}{2}+1<3,$$
which implies $\cup_{1\leq j\leq n}D(\nu_j,1/2)\subset D(z,3)$. Since the sets $D(\nu_j,1/2)$ are mutually disjoint, we obtain $n\leq36$. Therefore,
$$\mu(D(z,1))^{\gamma}=\left(\sum_{\nu\in\Gamma_z}\mu\big(D(z,1)\cap Q_1(\nu)\big)\right)^{\gamma}\lesssim
\sum_{\nu\in\Gamma_z}\mu(Q_1(\nu))^{\gamma}.$$
Consequently,
\begin{align*}
\int_{\C}\frac{\mu(D(z,1))^{\gamma}}{w(D(z,1))^{\eta}}dA(z)
&\lesssim\int_{\C}\sum_{\nu\in\Gamma_z}\frac{\mu(Q_1(\nu))^{\gamma}}{w(D(z,1))^{\eta}}dA(z)\\
&=\sum_{\nu\in\Z^2}\int_{\{z\in\C:D(z,1)\cap Q_1(\nu)\neq\emptyset\}}
  \frac{\mu(Q_1(\nu))^{\gamma}}{w(D(z,1))^{\eta}}dA(z)\\
&\leq\sum_{\nu\in\Z^2}\int_{D(\nu,2)}\frac{\mu(Q_1(\nu))^{\gamma}}{w(D(z,1))^{\eta}}dA(z)\\
&\asymp\sum_{\nu\in\Z^2}\frac{\mu(Q_1(\nu))^{\gamma}}{w(Q_1(\nu))^{\eta}},
\end{align*}
which finishes the proof.
\end{proof}
\end{lemma}

\section{Proof of Theorem \ref{main}}\label{proof}

In this section, we give the proof of Theorem \ref{main}.

\subsection{The case $1<p<2$}
Before proceeding, we establish several preliminary lemmas.

For $\alpha>0$ and two measurable functions $f,g$ on $\C$, define the pairing $\langle f,g\rangle_{\alpha}$ by
$$\langle f,g\rangle_{\alpha}:=\int_{\C}f(z)\overline{g(z)}e^{-\alpha|z|^2}dA(z).$$
The following lemma characterizes the dual space of $F^p_{\alpha,w}$ for $p>1$, which has its own interest.

\begin{lemma}\label{dual}
Let $\alpha>0$, $p>1$ and $w\in A^{\res}_p$. Then the dual space of $F^p_{\alpha,w}$ can be identified with $F^{p'}_{\alpha,w'}$ under the pairing $\langle \cdot,\cdot\rangle_{\alpha}$, where $w'=w^{-p'/p}$.
\end{lemma}
\begin{proof}
For $f\in F^p_{\alpha,w}$ and $g\in F^{p'}_{\alpha,w'}$, H\"{o}lder's inequality clearly implies
$$|\langle f,g\rangle_{\alpha}|\leq\|f\|_{F^p_{\alpha,w}}\|g\|_{F^{p'}_{\alpha,w'}},$$
which gives $F^{p'}_{\alpha,w'}\subset\left(F^p_{\alpha,w}\right)^*$. Conversely, assume that $\xi\in\left(F^p_{\alpha,w}\right)^*$. By the Hahn--Banach theorem, there exists $\hat{\xi}\in\left(L^p_{\alpha,w}\right)^*$ such that $\hat{\xi}|_{F^p_{\alpha,w}}=\xi$ and $\|\hat{\xi}\|=\|\xi\|$. Noting that $\left(L^p_{\alpha,w}\right)^*=L^{p'}_{\alpha,w'}$ under the pairing $\langle\cdot,\cdot\rangle_{\alpha}$, we may find some $g\in L^{p'}_{\alpha,w'}$ such that $\|g\|_{L^{p'}_{\alpha,w'}}\leq\|\hat{\xi}\|$ and for any $f\in F^p_{\alpha,w}$,
$$\xi(f)=\hat{\xi}(f)=\langle f,g\rangle_{\alpha}.$$
By the proof of \cite[Lemma 3.5]{CFP}, we know that $P_{\alpha}f=f$. Therefore,
$$\xi(f)=\langle f,g\rangle_{\alpha}=\langle P_{\alpha}f,g\rangle_{\alpha}
=\langle f,P_{\alpha}g\rangle_{\alpha},\quad f\in F^p_{\alpha,w}.$$
Since $w'\in A^{\res}_{p'}$ and consequently, $P_{\alpha}:L^{p'}_{\alpha,w'}\to F^{p'}_{\alpha,w'}$ is bounded (see \cite[Theorem 3.1]{Is}), we obtain that $P_{\alpha}g\in F^{p'}_{\alpha,w'}$ with $\|P_{\alpha}g\|_{F^{p'}_{\alpha,w'}}\lesssim\|g\|_{L^{p'}_{\alpha,w'}}\leq\|\xi\|$, which gives $\left(F^p_{\alpha,w}\right)^*\subset F^{p'}_{\alpha,w'}$ and finishes the proof.
\end{proof}

The following lemma gives an equivalent norm for the space $F^p_{\alpha,w}$ via the weight $\hat{w}$.

\begin{lemma}\label{hat}
Let $p,\alpha>0$ and $w\in A^{\res}_{\infty}$. Then $F^p_{\alpha,w}=F^p_{\alpha,\hat{w}}$ with equivalent norms.
\end{lemma}
\begin{proof}
We only prove the inclusion $F^p_{\alpha,\hat{w}}\subset F^p_{\alpha,w}$ and the inequality
$$\|f\|_{F^p_{\alpha,w}}\lesssim\|f\|_{F^p_{\alpha,\hat{w}}}.$$
The other direction is similar.

Suppose $f\in F^p_{\alpha,\hat{w}}$. For any $z\in\C$, using Lemma \ref{pointwise} to the weight $\hat{w}$ yields
$$|f(z)|^pe^{-\frac{p\alpha}{2}|z|^2}\lesssim
\frac{1}{\hat{w}(D(z,1))}\int_{D(z,1)}|f(u)|^pe^{-\frac{p\alpha}{2}|u|^2}\hat{w}(u)dA(u).$$
Consequently, by Lemma \ref{esti},
\begin{align*}
\|f\|^p_{F^p_{\alpha,w}}&=\sum_{\nu\in\Z^2}\int_{Q_1(\nu)}|f(z)|^pe^{-\frac{p\alpha}{2}|z|^2}w(z)dA(z)\\
&\lesssim\sum_{\nu\in\Z^2}\int_{Q_1(\nu)}\frac{w(z)}{\hat{w}(D(z,1))}
  \int_{D(z,1)}|f(u)|^pe^{-\frac{p\alpha}{2}|u|^2}\hat{w}(u)dA(u)dA(z)\\
&\leq\sum_{\nu\in\Z^2}\int_{Q_1(\nu)}\frac{w(z)}{\hat{w}(D(z,1))}dA(z)
  \int_{D(\nu,2)}|f(u)|^pe^{-\frac{p\alpha}{2}|u|^2}\hat{w}(u)dA(u)\\
&\lesssim\int_{\C}|f(u)|^pe^{-\frac{p\alpha}{2}|u|^2}\hat{w}(u)dA(u)=\|f\|^p_{F^p_{\alpha,\hat{w}}}.
\end{align*}
The proof is complete.
\end{proof}

We are now ready to characterize the $r$-summability for the Carleson embeddings. The first result concerns the case that the target space is a Hilbert space.

\begin{proposition}\label{2sum}
Let $\alpha>0$, $1<p<2$, $w\in A^{\res}_p$, and let $\mu$ be a positive Borel measure on $\C$. The following assertions are equivalent.
\begin{enumerate}[(a)]
  \item $I_d:F^p_{\alpha,w}\to L^2_{\alpha}(\mu)$ is $2$-summing.
  \item $\int_{\C}\left(\int_{\C}|e^{\alpha \bar{u}z}|^2e^{-\alpha|z|^2}d\mu(z)\right)^{p'/2}
      e^{-\frac{p'\alpha}{2}|u|^2}w(u)^{-p'/p}dA(u)<\infty.$
  \item $P_{\alpha}:L^p_{\alpha,w}\to L^2_{\alpha}(\mu)$ is $1$-summing.
  \item The sequence
  $$\lambda:=\left\{\frac{\mu(Q_1(\nu))}{w(Q_1(\nu))^{2/p}}\right\}_{\nu\in\Z^2}$$
  belongs to $l^{p'/2}(\Z^2)$.
\end{enumerate}

Moreover,
\begin{align*}
\pi_2\big(I_d:F^p_{\alpha,w}&\to L^2_{\alpha}(\mu)\big)\asymp\|\lambda\|^{1/2}_{l^{p'/2}(\Z^2)}\asymp\\
&\left(\int_{\C}\left(\int_{\C}|e^{\alpha \bar{u}z}|^2e^{-\alpha|z|^2}d\mu(z)\right)^{p'/2}
      e^{-\frac{p'\alpha}{2}|u|^2}w(u)^{-p'/p}dA(u)\right)^{1/p'}.
\end{align*}
\end{proposition}
\begin{proof}
We first establish the chain of implications:
$$\mathrm{(a)}\Longrightarrow\mathrm{(b)}\Longrightarrow\mathrm{(c)}\Longrightarrow\mathrm{(a)}.$$
To this end, assume that (a) holds. Then $I_d:F^p_{\alpha,w}\to L^2_{\alpha}(\mu)$ is $p'$-summing since $1<p<2$. For any $u\in\C$, consider the reproducing kernel $K_u(z)=e^{\alpha\bar{u}z}$, which belongs to $F^p_{\alpha,w}$ thanks to Lemma \ref{test}. By Lemma \ref{dual}, for any $\xi\in B_{(F^p_{\alpha,w})^*}$, there exists $g_{\xi}\in F^{p'}_{\alpha,w'}$ with $\|g_{\xi}\|_{F^{p'}_{\alpha,w'}}\lesssim1$ such that
$$\xi(K_u)=\langle K_u,g_{\xi}\rangle_{\alpha}=\frac{\pi}{\alpha}\overline{P_{\alpha}(g_{\xi})(u)}
=\frac{\pi}{\alpha}\overline{g_{\xi}(u)},\quad u\in\C.$$
Therefore, Theorem \ref{sum-T} implies that
\begin{align*}
&\left(\int_{\C}\|K_u\|^{p'}_{L^2_{\alpha}(\mu)}e^{-\frac{p'\alpha}{2}|u|^2}w'(u)dA(u)\right)^{1/p'}\\
&\ \leq\pi_{p'}\left(I_d:F^p_{\alpha,w}\to L^2_{\alpha}(\mu)\right)\cdot
  \sup_{\xi\in B_{(F^p_{\alpha,w})^*}}
  \left(\int_{\C}|\xi(K_u)|^{p'}e^{-\frac{p'\alpha}{2}|u|^2}w'(u)dA(u)\right)^{1/p'}\\
&\ \lesssim\pi_2\left(I_d:F^p_{\alpha,w}\to L^2_{\alpha}(\mu)\right)\cdot
  \sup_{\xi\in B_{(F^p_{\alpha,w})^*}}
  \left(\int_{\C}|g_{\xi}(u)|^{p'}e^{-\frac{p'\alpha}{2}|u|^2}w'(u)dA(u)\right)^{1/p'}\\
&\ \lesssim\pi_2\left(I_d:F^p_{\alpha,w}\to L^2_{\alpha}(\mu)\right).
\end{align*}
Therefore, we obtain
\begin{align*}
&\left(\int_{\C}\left(\int_{\C}|e^{\alpha \bar{u}z}|^2e^{-\alpha|z|^2}d\mu(z)\right)^{p'/2}
      e^{-\frac{p'\alpha}{2}|u|^2}w(u)^{-p'/p}dA(u)\right)^{1/p'}\\
      &\qquad\qquad\qquad\qquad\qquad\qquad\lesssim\pi_2\left(I_d:F^p_{\alpha,w}\to L^2_{\alpha}(\mu)\right),
\end{align*}
that is, (b) holds.

Assume (b) holds. For any $f\in L^2_{\alpha}(\mu)$, define
$$R_{\alpha}f(u):=\frac{\alpha}{\pi}\int_{\C}f(z)e^{\alpha u\bar{z}}e^{-\alpha|z|^2}d\mu(z),\quad u\in\C.$$
We claim that $R_{\alpha}:L^2_{\alpha}(\mu)\to L^{p'}_{\alpha,w'}$ is order bounded. In fact, for any $f\in B_{L^2_{\alpha}(\mu)}$, the Cauchy--Schwarz inequality yields
$$|R_{\alpha}f(u)|\lesssim\left(\int_{\C}|e^{\alpha\bar{u}z}|^2e^{-\alpha|z|^2}d\mu(z)\right)^{1/2},\quad u\in\C.$$
Hence by (b), $\sup_{f\in B_{L^2_{\alpha}(\mu)}}|R_{\alpha}f|\in L^{p'}_{\alpha,w'}$, which implies that $R_{\alpha}:L^2_{\alpha}(\mu)\to L^{p'}_{\alpha,w'}$ is order bounded. Consequently, $R_{\alpha}:L^2_{\alpha}(\mu)\to L^{p'}_{\alpha,w'}$ is $p'$-summing due to Lemma \ref{order}. Moreover,
\begin{align*}
\pi_{p'}\big(R_{\alpha}:L^2_{\alpha}(\mu)&\to L^{p'}_{\alpha,w'}\big)\lesssim\\
&\left(\int_{\C}\left(\int_{\C}|e^{\alpha \bar{u}z}|^2e^{-\alpha|z|^2}d\mu(z)\right)^{p'/2}
      e^{-\frac{p'\alpha}{2}|u|^2}w(u)^{-p'/p}dA(u)\right)^{1/p'}.
\end{align*}
On the other hand, it is easy to verify that $R^*_{\alpha}=P_{\alpha}$ under the pairings $\langle\cdot,\cdot\rangle_{\alpha}$ and $\langle\cdot,\cdot\rangle_{L^2_{\alpha}(\mu)}$. Hence we know that $P^*_{\alpha}=R^{**}_{\alpha}=R_{\alpha}$ is $p'$-summing, which, in conjunction with \cite[Theorem 2.21]{DJT}, implies that $P_{\alpha}:L^p_{\alpha,w}\to L^2_{\alpha}(\mu)$ is $1$-summing with
\begin{align*}
\pi_1\big(P_{\alpha}:L^p_{\alpha,w}&\to L^2_{\alpha}(\mu)\big)\lesssim
 \pi_{p'}\big(R_{\alpha}:L^2_{\alpha}(\mu)\to L^{p'}_{\alpha,w'}\big)\lesssim\\
&\left(\int_{\C}\left(\int_{\C}|e^{\alpha \bar{u}z}|^2e^{-\alpha|z|^2}d\mu(z)\right)^{p'/2}
      e^{-\frac{p'\alpha}{2}|u|^2}w(u)^{-p'/p}dA(u)\right)^{1/p'}.
\end{align*}
Therefore, (c) holds.

Assume that (c) holds. It is known that for any $f\in F^p_{\alpha,w}$, $P_{\alpha}f=f$ (see the proof of \cite[Lemma 3.5]{CFP}). Consequently, we have the following commutative diagram:
\begin{equation*}
\xymatrix{
  F^p_{\alpha,w} \ar[d]_{I_d} \ar[dr]^{I_d}        \\
  L^p_{\alpha,w} \ar[r]_{P_{\alpha}}  & L^2_{\alpha}(\mu).    }
\end{equation*}
Therefore, using the ideal property of $1$-summing operators and (c), we obtain that $I_d:F^p_{\alpha,w}\to L^2_{\alpha}(\mu)$ is $1$-summing, and
\begin{align*}
\pi_1\big(I_d:F^p_{\alpha,w}\to L^2_{\alpha}(\mu)\big)\leq
  \pi_1\big(P_{\alpha}:L^p_{\alpha,w}\to L^2_{\alpha}(\mu)\big),
\end{align*}
which gives (a).

We next establish the equivalence (b)$\Longleftrightarrow$(d). Before proceeding, we apply Lemma \ref{esti} to find a constant $C>0$ such that for any $\nu,\nu'\in\mathbb{Z}^2$,
\begin{equation}\label{wQ}
C^{-|\nu-\nu'|}\leq\left(\frac{w(Q_1(\nu'))}{w(Q_1(\nu))}\right)^{2/p}\leq C^{|\nu-\nu'|}.
\end{equation}
Then, using \eqref{w'}, \eqref{wQ} and H\"{o}lder's inequality, we obtain that
\begin{align*}
&\int_{\C}\left(\int_{\C}|e^{\alpha \bar{u}z}|^2e^{-\alpha|z|^2}d\mu(z)\right)^{p'/2}
  e^{-\frac{p'\alpha}{2}|u|^2}w(u)^{-p'/p}dA(u)\\
&\ =\int_{\C}\left(\int_{\C}e^{-\alpha|z-u|^2}d\mu(z)\right)^{p'/2}w(u)^{-p'/p}dA(u)\\
&\ \lesssim\sum_{\nu\in\mathbb{Z}^2}\left(\sum_{\nu'\in\mathbb{Z}^2}e^{-\frac{\alpha}{2}|\nu-\nu'|^2}\mu(Q_1(\nu'))\right)^{p'/2}
  \cdot w'(Q_1(\nu))\\
&\ \asymp\sum_{\nu\in\mathbb{Z}^2}\left(\sum_{\nu'\in\mathbb{Z}^2}e^{-\frac{\alpha}{2}|\nu-\nu'|^2}\mu(Q_1(\nu'))\right)^{p'/2}
  \cdot w(Q_1(\nu))^{-p'/p}\\
&\ \leq\sum_{\nu\in\mathbb{Z}^2}\left(\sum_{\nu'\in\mathbb{Z}^2}e^{-\frac{\alpha}{2}|\nu-\nu'|^2}C^{|\nu-\nu'|}
  \frac{\mu(Q_1(\nu'))}{w(Q_1(\nu'))^{2/p}}\right)^{p'/2}\\
&\ \lesssim\sum_{\nu\in\mathbb{Z}^2}\sum_{\nu'\in\mathbb{Z}^2}e^{-\frac{p'\alpha}{8}|\nu-\nu'|^2}
  \left(\frac{\mu(Q_1(\nu'))}{w(Q_1(\nu'))^{2/p}}\right)^{p'/2}\\
&\ \asymp\sum_{\nu'\in\mathbb{Z}^2}\left(\frac{\mu(Q_1(\nu'))}{w(Q_1(\nu'))^{2/p}}\right)^{p'/2}.
\end{align*}
Conversely, noting that $p'/2>1$, we have that
\begin{align*}
&\int_{\C}\left(\int_{\C}|e^{\alpha \bar{u}z}|^2e^{-\alpha|z|^2}d\mu(z)\right)^{p'/2}
e^{-\frac{p'\alpha}{2}|u|^2}w(u)^{-p'/p}dA(u)\\
&\ \gtrsim\sum_{\nu\in\mathbb{Z}^2}\left(\sum_{\nu'\in\mathbb{Z}^2}e^{-2\alpha|\nu-\nu'|^2}\mu(Q_1(\nu'))\right)^{p'/2}
  \cdot w'(Q_1(\nu))\\
&\ \gtrsim\sum_{\nu\in\mathbb{Z}^2}\left(\sum_{\nu'\in\mathbb{Z}^2}e^{-2\alpha|\nu-\nu'|^2}C^{-|\nu-\nu'|}
  \frac{\mu(Q_1(\nu'))}{w(Q_1(\nu'))^{2/p}}\right)^{p'/2}\\
&\ \geq\sum_{\nu\in\mathbb{Z}^2}\sum_{\nu'\in\mathbb{Z}^2}e^{-p'\alpha|\nu-\nu'|^2}C^{-\frac{p'}{2}|\nu-\nu'|}
  \left(\frac{\mu(Q_1(\nu'))}{w(Q_1(\nu'))^{2/p}}\right)^{p'/2}\\
&\ \asymp\sum_{\nu'\in\mathbb{Z}^2}\left(\frac{\mu(Q_1(\nu'))}{w(Q_1(\nu'))^{2/p}}\right)^{p'/2}.
\end{align*}
Therefore, the equivalence (b)$\Longleftrightarrow$(d) follows, and
$$\int_{\C}\left(\int_{\C}|e^{\alpha \bar{u}z}|^2e^{-\alpha|z|^2}d\mu(z)\right)^{p'/2}
e^{-\frac{p'\alpha}{2}|u|^2}w(u)^{-p'/p}dA(u)\asymp\|\lambda\|^{p'/2}_{l^{p'/2}(\mathbb{Z}^2)}.$$
The proof is finished.
\end{proof}

For $\alpha>0$, the Berezin transform $B_{\alpha}$ is defined for $f\in L^1_{\mathrm{loc}}(\C,dA)$ by
$$B_{\alpha}f(z):=\int_{\C}f(u)e^{-\alpha|z-u|^2}dA(u),\quad z\in\C.$$
The main result of this subsection is as follows, which characterizes the $r$-summing embeddings $I_d:F^p_{\alpha,w}\to L^q_{\alpha}(\mu)$ for all $1<p<2$, $1\leq q\leq2$ and $r\geq1$.

\begin{theorem}\label{<}
Let $\alpha>0$, $1<p<2$, $1\leq q\leq2$, $w\in A^{\res}_p$, and let $\mu$ be a positive Borel measure on $\C$. The following assertions are equivalent.
\begin{enumerate}[(a)]
  \item $I_d:F^p_{\alpha,w}\to L^q_{\alpha}(\mu)$ is $2$-summing.
  \item $I_d:F^p_{\alpha,\hat{w}}\to L^q_{\alpha}(\mu)$ is $2$-summing.
  \item $I_d:F^p_{\alpha,w}\to L^q_{\alpha}(\mu)$ is $r$-summing for any $r\geq1$.
  \item $B_{\alpha}:L^{\frac{p}{2-p}}(\hat{w}^{2/(2-p)}dA)\to L^{q/2}(\mu)$ is bounded.
  \item The sequence
  $$\lambda:=
  \left\{\frac{\mu(Q_1(\nu))}{w(Q_1(\nu))^{q/p}}\right\}_{\nu\in\Z^2}$$
  belongs to $l^s(\Z^2)$, where $s=2p/(2p-2q+pq)$.
\end{enumerate}

Moreover, if one of these conditions holds, then
\begin{equation}\label{pi2}
\pi_2\big(I_d:F^p_{\alpha,w}\to L^q_{\alpha}(\mu)\big)\asymp\|\lambda\|^{1/q}_{l^s(\Z^2)}.
\end{equation}
\end{theorem}
\begin{proof}
It follows from Lemma \ref{hat} that the conditions (a) and (b) are equivalent, and
\begin{equation}\label{2hat}
\pi_2\big(I_d:F^p_{\alpha,w}\to L^q_{\alpha}(\mu)\big)\asymp
\pi_2\big(I_d:F^p_{\alpha,\hat{w}}\to L^q_{\alpha}(\mu)\big).
\end{equation}
Since $p,q\leq2$ and consequently, $F^p_{\alpha,w}$ and $L^q_{\alpha}(\mu)$ both have cotype $2$, we know that $\Pi_r\big(F^p_{\alpha,w},L^q_{\alpha}(\mu)\big)=\Pi_1\big(F^p_{\alpha,w},L^q_{\alpha}(\mu)\big)$ for any $r>1$ (see \cite[Corollary 11.16]{DJT}). Hence (a) and (c) are equivalent, and
$$\pi_2\big(I_d:F^p_{\alpha,w}\to L^q_{\alpha}(\mu)\big)\asymp
\pi_r\big(I_d:F^p_{\alpha,w}\to L^q_{\alpha}(\mu)\big).$$

We now concentrate on the equivalence of (d) and (e). Assume that (d) holds. For any $c=\{c_{\nu}\}\in l^{\frac{p}{2-p}}(\Z^2)$ and any $\tau\in[0,1]$, define
$$G_{\tau}(z)=\sum_{\nu\in\Z^2}c_{\nu}\delta_{\nu}(\tau)\hat{w}(z)^{-2/p}\chi_{Q_1(\nu)}(z),\quad z\in\C,$$
where $\{\delta_{\nu}\}_{\nu\in\Z^2}$ is a sequence of Rademacher functions on $[0,1]$. Then it is easy to see that for almost every $\tau\in[0,1]$, $G_{\tau}\in L^{\frac{p}{2-p}}(\hat{w}^{2/(2-p)}dA)$, and
$$\|G_{\tau}\|_{L^{\frac{p}{2-p}}(\hat{w}^{2/(2-p)}dA)}=\|c\|_{l^{\frac{p}{2-p}}(\Z^2)}.$$
Moreover,
$$B_{\alpha}G_{\tau}(z)=\sum_{\nu\in\Z^2}c_{\nu}\delta_{\nu}(\tau)
\int_{Q_1(\nu)}e^{-\alpha|z-u|^2}\hat{w}(u)^{-2/p}dA(u).$$
Since $B_{\alpha}:L^{\frac{p}{2-p}}(\hat{w}^{2/(2-p)}dA)\to L^{q/2}(\mu)$ is bounded, we have
$$\int_{\C}\left|\sum_{\nu\in\Z^2}c_{\nu}\delta_{\nu}(\tau)
\int_{Q_1(\nu)}e^{-\alpha|z-u|^2}\hat{w}(u)^{-2/p}dA(u)\right|^{q/2}d\mu(z)
\leq\|B_{\alpha}\|^{q/2}\|c\|^{q/2}_{l^{\frac{p}{2-p}}(\Z^2)}.$$
Integrating with respect to $\tau$ on $[0,1]$ and using Fubini's theorem, we get
\begin{align*}
\int_{\C}\int_0^1\left|\sum_{\nu\in\Z^2}c_{\nu}\delta_{\nu}(\tau)
\int_{Q_1(\nu)}e^{-\alpha|z-u|^2}\hat{w}(u)^{-2/p}dA(u)\right|^{q/2}d\tau d\mu(z)\\
\leq\|B_{\alpha}\|^{q/2}\|c\|^{q/2}_{l^{\frac{p}{2-p}}(\Z^2)}.
\end{align*}
Further, using Khinchine's inequality yields that
\begin{align*}
&\|B_{\alpha}\|^{q/2}\|c\|^{q/2}_{l^{\frac{p}{2-p}}(\Z^2)}\\
&\ \gtrsim\int_{\C}\left(\sum_{\nu\in\Z^2}\left|c_{\nu}\right|^2
  \left(\int_{Q_1(\nu)}e^{-\alpha|z-u|^2}\hat{w}(u)^{-2/p}dA(u)\right)^2\right)^{q/4}d\mu(z)\\
&\ \geq\int_{\C}\left(\sum_{\nu\in\Z^2}\left|c_{\nu}\right|^2
\left(\int_{Q_1(\nu)}e^{-\alpha|z-u|^2}\hat{w}(u)^{-2/p}dA(u)\right)^2\chi_{Q_1(\nu)}(z)\right)^{q/4}d\mu(z)\\
&\ =\sum_{\nu\in\Z^2}\left|c_{\nu}\right|^{q/2}
  \int_{Q_1(\nu)}\left(\int_{Q_1(\nu)}e^{-\alpha|z-u|^2}\hat{w}(u)^{-2/p}dA(u)\right)^{q/2}d\mu(z)\\
&\ \asymp\sum_{\nu\in\Z^2}\left|c_{\nu}\right|^{q/2}\frac{\mu(Q_1(\nu))}{w(Q_1(\nu))^{q/p}}.
\end{align*}
Since $c\in l^{\frac{p}{2-p}}(\Z^2)$, we know that $\left\{\left|c_{\nu}\right|^{q/2}\right\}\in l^{\frac{2p}{q(2-p)}}(\Z^2)$ with $$\left\|\left\{\left|c_{\nu}\right|^{q/2}\right\}\right\|_{l^{\frac{2p}{q(2-p)}}(\Z^2)}
=\|c\|^{q/2}_{l^{\frac{p}{2-p}}(\Z^2)}.$$
Hence by the duality, we obtain that
$$\lambda=\left\{\frac{\mu(Q_1(\nu))}{w(Q_1(\nu))^{q/p}}\right\}
\in\left(l^{\frac{2p}{q(2-p)}}(\Z^2)\right)^*=l^s(\Z^2)$$
with
\begin{equation}\label{Blow}
\|\lambda\|_{l^s(\Z^2)}\lesssim\|B_{\alpha}\|^{q/2}.
\end{equation}
Conversely, assume that (e) holds. Then for any $f\in L^{\frac{p}{2-p}}(\hat{w}^{2/(2-p)}dA)$, using H\"{o}lder's inequality with respect to the conjugate exponents $\frac{2p}{q(2-p)}$ and $\left(\frac{2p}{q(2-p)}\right)'=s$, we establish that
\begin{align*}
\|B_{\alpha}f\|^{q/2}_{L^{q/2}(\mu)}
&=\int_{\C}\left|\int_{\C}f(u)e^{-\alpha|z-u|^2}dA(u)\right|^{q/2}d\mu(z)\\
&\leq\sum_{\nu\in\Z^2}\int_{Q_1(\nu)}\left(\sum_{\nu'\in\Z^2}\int_{Q_1(\nu')}
  |f(u)|e^{-\alpha|z-u|^2}dA(u)\right)^{q/2}d\mu(z)\\
&\lesssim\sum_{\nu\in\Z^2}\int_{Q_1(\nu)}\left(\sum_{\nu'\in\Z^2}e^{-\frac{\alpha}{2}|\nu-\nu'|^2}
  \int_{Q_1(\nu')}|f(u)|dA(u)\right)^{q/2}d\mu(z)\\
&\leq\sum_{\nu,\nu'\in\Z^2}e^{-\frac{q\alpha}{4}|\nu-\nu'|^2}
  \int_{Q_1(\nu)}\left(\int_{Q_1(\nu')}|f|dA\right)^{q/2}d\mu(z)\\
&=\sum_{\nu,\nu'\in\Z^2}e^{-\frac{q\alpha}{4}|\nu-\nu'|^2}
  \left(\int_{Q_1(\nu')}|f|\hat{w}(\nu)^{2/p}dA\right)^{q/2}\frac{\mu(Q_1(\nu))}{w(Q_1(\nu))^{q/p}}\\
&\leq\left(\sum_{\nu,\nu'\in\Z^2}e^{-\frac{q\alpha}{4}|\nu-\nu'|^2}
  \left(\int_{Q_1(\nu')}|f|\hat{w}(\nu)^{2/p}dA\right)^{\frac{p}{2-p}}\right)^{\frac{q(2-p)}{2p}}\\
&\qquad\qquad\cdot\left(\sum_{\nu,\nu'\in\Z^2}e^{-\frac{q\alpha}{4}|\nu-\nu'|^2}
  \left(\frac{\mu(Q_1(\nu))}{w(Q_1(\nu))^{q/p}}\right)^s\right)^{1/s}\\
&=:\mathcal{S}_1\cdot\mathcal{S}_2.
\end{align*}
It is clear that
$$\mathcal{S}_2=\left(\sum_{\nu,\nu'\in\Z^2}e^{-\frac{q\alpha}{4}|\nu-\nu'|^2}
  \left(\frac{\mu(Q_1(\nu))}{w(Q_1(\nu))^{q/p}}\right)^s\right)^{1/s}\lesssim\|\lambda\|_{l^s(\Z^2)}.$$
To estimate the term $\mathcal{S}_1$, we apply the first part of Lemma \ref{esti} to find some $C>0$ such that for any $\nu,\nu'\in\Z^2$, $\hat{w}(\nu)^{2/(2-p)}\leq C^{|\nu-\nu'|}\hat{w}(\nu')^{2/(2-p)}$. Then, noting that $p/(2-p)>1$, and applying H\"{o}lder's inequality together with the second part of Lemma \ref{esti}, we obtain that
\begin{align*}
\mathcal{S}_1&\leq\left(\sum_{\nu,\nu'\in\Z^2}e^{-\frac{q\alpha}{4}|\nu-\nu'|^2}
  \int_{Q_1(\nu')}|f|^{\frac{p}{2-p}}\hat{w}(\nu)^{\frac{2}{2-p}}dA\right)^{\frac{q(2-p)}{2p}}\\
&\leq\left(\sum_{\nu,\nu'\in\Z^2}e^{-\frac{q\alpha}{4}|\nu-\nu'|^2}C^{|\nu-\nu'|}
  \int_{Q_1(\nu')}|f|^{\frac{p}{2-p}}\hat{w}(\nu')^{\frac{2}{2-p}}dA\right)^{\frac{q(2-p)}{2p}}\\
&\asymp\left(\sum_{\nu,\nu'\in\Z^2}e^{-\frac{q\alpha}{4}|\nu-\nu'|^2}C^{|\nu-\nu'|}
  \int_{Q_1(\nu')}|f|^{\frac{p}{2-p}}\hat{w}^{\frac{2}{2-p}}dA\right)^{\frac{q(2-p)}{2p}}\\
&\lesssim\|f\|^{q/2}_{L^{\frac{p}{2-p}}(\hat{w}^{2/(2-p)}dA)}.
\end{align*}
Hence $B_{\alpha}:L^{\frac{p}{2-p}}(\hat{w}^{2/(2-p)}dA)\to L^{q/2}(\mu)$ is bounded, and
\begin{equation}\label{Bupp}
\|B_{\alpha}\|^{q/2}\lesssim\|\lambda\|_{l^s(\Z^2)}.
\end{equation}
Combining \eqref{Blow} with \eqref{Bupp}, we get that (d) and (e) are equivalent, and
\begin{equation}\label{Bequi}
\|B_{\alpha}\|^{q/2}\asymp\|\lambda\|_{l^s(\Z^2)}.
\end{equation}

In the case $q=2$, it follows from Proposition \ref{2sum} that the conditions (a) and (e) are equivalent. Therefore, we will complete the proof by establishing the equivalence of (b) and (d) in the case $q<2$. To this end, we claim that if (b) or (d) holds, then the embedding $I_d:F^p_{\alpha,\hat{w}}\to L^q_{\alpha}(\mu)$ is bounded. In fact, if (b) holds, then the desired boundedness is clear. Suppose that (d) holds. Equivalently, (e) holds. Then in the case $p\leq q$, we have $\lambda\in l^s(\Z^2)\subset l^{\infty}(\Z^2)$, which implies that $I_d:F^p_{\alpha,w}\to L^q_{\alpha}(\mu)$ is bounded (see the proof of \cite[Theorem 4.3]{CFP}), and consequently, $I_d:F^p_{\alpha,\hat{w}}\to L^q_{\alpha}(\mu)$ is bounded. In the case $p>q$, noting that $s=\frac{2p}{2p-2q+pq}<\frac{p}{p-q}$, we have $\lambda\in l^s(\Z^2)\subset l^{\frac{p}{p-q}}(\Z^2)$. Therefore, By Lemmas \ref{esti} and \ref{dis},
\begin{align*}
\int_{\C}\left(\frac{\mu(D(z,1))}{\hat{w}(D(z,1))}\right)^{\frac{p}{p-q}}\hat{w}(z)dA(z)
&\asymp\int_{\C}\left(\frac{\mu(D(z,1))}{w(D(z,1))^{q/p}}\right)^{\frac{p}{p-q}}dA(z)\\
&\asymp\sum_{\nu\in\Z^2}\left(\frac{\mu(Q_1(\nu))}{w(Q_1(\nu))^{q/p}}\right)^{\frac{p}{p-q}}<\infty,
\end{align*}
which, in conjunction with \cite[Theorem 4.4]{CFP}, implies that $I_d:F^p_{\alpha,\hat{w}}\to L^q_{\alpha}(\mu)$ is bounded. Hence the claim holds. Write now that $k=p/(2-p)$, $t=q/(2-q)$, and $d\mu_{\alpha,q}(z)=e^{-\frac{q\alpha}{2}|z|^2}d\mu(z)$. If one of (b) and (d) holds, then by the boundedness of $I_d:F^p_{\alpha,\hat{w}}\to L^q_{\alpha}(\mu)$, we may apply \cite[Lemma 7.6]{LR18} and Proposition \ref{2sum} to establish that
\begin{align*}
&\pi_2\big(I_d:F^p_{\alpha,\hat{w}}\to L^q_{\alpha}(\mu)\big)^2\\
&\ =\pi_2\big(I_d:F^p_{\alpha,\hat{w}}\to L^q(\mu_{\alpha,q})\big)^2\\
&\ \asymp\inf_{F\in B^+_{L^{2t}(\mu_{\alpha,q})}}
  \pi_2\big(I_d:F^p_{\alpha,\hat{w}}\to L^2(F^{-2}d\mu_{\alpha,q})\big)^2\\
&\ \asymp\inf_{F\in B^+_{L^{2t}(\mu_{\alpha,q})}}
  \left(\int_{\C}\left(\int_{\C}|e^{\alpha\bar{u}z}|^2
  \frac{e^{-\frac{q\alpha}{2}|z|^2}}{F(z)^2}d\mu(z)\right)^{p'/2}
  e^{-\frac{p'\alpha}{2}|u|^2}\hat{w}(u)^{-p'/p}dA(u)\right)^{2/p'}\\
&\ =\inf_{f\in B^+_{L^{t}(\mu_{\alpha,q})}}
  \left(\int_{\C}\left(\int_{\C}|e^{\alpha\bar{u}z}|^2
  \frac{e^{-\frac{q\alpha}{2}|z|^2}}{f(z)}d\mu(z)\right)^{p'/2}
  e^{-\frac{p'\alpha}{2}|u|^2}\hat{w}(u)^{-p'/p}dA(u)\right)^{2/p'}\\
&\ =\inf_{f\in B^+_{L^{t}(\mu_{\alpha,q})}}
  \left\|u\mapsto\int_{\C}|e^{\alpha\bar{u}z}|^2f(z)^{-1}e^{-\frac{q\alpha}{2}|z|^2}d\mu(z)
  \right\|_{L^{p'/2}_{2\alpha,\hat{w}^{-p'/p}}}.
\end{align*}
Since $\left(L^{p'/2}_{2\alpha,\hat{w}^{-p'/p}}\right)^*=L^k_{2\alpha,\hat{w}^{2/(2-p)}}$ under the pairing $\langle\cdot,\cdot\rangle_{2\alpha}$, we obtain
\begin{align*}
&\pi_2\big(I_d:F^p_{\alpha,\hat{w}}\to L^q_{\alpha}(\mu)\big)^2\\
&\ \asymp\inf_{f\in B^+_{L^{t}(\mu_{\alpha,q})}}\sup_{g\in B^+_{L^k_{2\alpha,\hat{w}^{2/(2-p)}}}}
  \int_{\C}\int_{\C}|e^{\alpha\bar{u}z}|^2f(z)^{-1}e^{-\frac{q\alpha}{2}|z|^2}d\mu(z)
  g(u)e^{-2\alpha|u|^2}dA(u).
\end{align*}
Using Ky Fan's lemma as in \cite[p. 579]{LR18}, we get that
\begin{align*}
&\pi_2\big(I_d:F^p_{\alpha,\hat{w}}\to L^q_{\alpha}(\mu)\big)^2\\
&\ \asymp\sup_{g\in B^+_{L^k_{2\alpha,\hat{w}^{2/(2-p)}}}}\inf_{f\in B^+_{L^{t}(\mu_{\alpha,q})}}
  \int_{\C}\int_{\C}|e^{\alpha\bar{u}z}|^2f(z)^{-1}e^{-\frac{q\alpha}{2}|z|^2}d\mu(z)
  g(u)e^{-2\alpha|u|^2}dA(u).
\end{align*}
Note that $t/(t+1)=q/2$. Combining Fubini's theorem with \cite[Lemma 7.5]{LR18} yields that
\begin{align*}
&\pi_2\big(I_d:F^p_{\alpha,\hat{w}}\to L^q_{\alpha}(\mu)\big)^2\\
&\ \asymp\sup_{g\in B^+_{L^k_{2\alpha,\hat{w}^{2/(2-p)}}}}\inf_{f\in B^+_{L^{t}(\mu_{\alpha,q})}}
  \int_{\C}\int_{\C}|e^{\alpha\bar{u}z}|^2g(u)e^{-2\alpha|u|^2}dA(u)
  f(z)^{-1}d\mu_{\alpha,q}(z)\\
&\ =\sup_{g\in B^+_{L^k_{2\alpha,\hat{w}^{2/(2-p)}}}}
  \left(\int_{\C}\left(\int_{\C}g(u)|e^{\alpha\bar{u}z}|^2e^{-2\alpha|u|^2}dA(u)\right)^{q/2}
  e^{-\frac{q\alpha}{2}|z|^2}d\mu(z)\right)^{2/q}\\
&\ =\sup_{g\in B^+_{L^k_{2\alpha,\hat{w}^{2/(2-p)}}}}
  \left(\int_{\C}\left(\int_{\C}g(u)e^{-\alpha|u|^2}e^{-\alpha|z-u|^2}dA(u)\right)^{q/2}
  d\mu(z)\right)^{2/q}.
\end{align*}
Write $\tilde{g}(z)=g(z)e^{-\alpha|z|^2}$. Then
$$\left\|\tilde{g}\right\|_{L^k(\hat{w}^{2/(2-p)}dA)}=\|g\|_{L^k_{2\alpha,\hat{w}^{2/(2-p)}}},$$
which gives that
\begin{align}\label{fin}
\nonumber&\pi_2\big(I_d:F^p_{\alpha,\hat{w}}\to L^q_{\alpha}(\mu)\big)^2\\
&\nonumber\ \asymp\sup_{\tilde{g}\in B^+_{L^k(\hat{w}^{2/(2-p)}dA)}}
  \left(\int_{\C}\left(\int_{\C}\tilde{g}(u)e^{-\alpha|z-u|^2}dA(u)\right)^{q/2}
  d\mu(z)\right)^{2/q}\\
&\ =\sup_{\tilde{g}\in B^+_{L^k(\hat{w}^{2/(2-p)}dA)}}\left\|B_{\alpha}\tilde{g}\right\|_{L^{q/2}(\mu)}
  =\|B_{\alpha}\|_{L^k(\hat{w}^{2/(2-p)}dA)\to L^{q/2}(\mu)}.
\end{align}
Therefore, (b) and (d) are equivalent in the case $q<2$. Combining \eqref{2hat} with \eqref{Bequi} and \eqref{fin}, we obtain the estimate \eqref{pi2} for the case $q<2$. In the case $q=2$, \eqref{pi2} was established in Proposition \ref{2sum}. The proof is complete.
\end{proof}

As a direct consequence of Theorem \ref{<} and Lemma \ref{dis}, we obtain the $1<p<2$ part of Theorem \ref{main}.

\begin{corollary}
Let $\alpha>0$, $1<p<2$, $w\in A^{\res}_p$, and let $\mu$ be a positive Borel measure on $\C$. Then for any $r\geq1$, $I_d:F^p_{\alpha,w}\to L^p_{\alpha}(\mu)$ is $r$-summing if and only if $\widehat{\mu}_w\in L^{2/p}$. Moreover,
$$\pi_r(I_d)\asymp\pi_2(I_d)\asymp\left(\int_{\C}\widehat{\mu}_w(z)^{2/p}dA(z)\right)^{1/2}.$$
\end{corollary}

\subsection{The case $p\geq2$}

Throughout this subsection, we fix a sequence $\lambda$ by
\begin{equation*}
\lambda=\{\lambda_{\nu}\}_{\nu\in\Z^2}:=
\left\{\left(\frac{\mu(Q_1(\nu))}{w(Q_1(\nu))}\right)^{1/p}\right\}_{\nu\in\Z^2}.
\end{equation*}

We establish a connection between the $r$-summability of $M_{\lambda}:l^p(\Z^2)\to l^p(\Z^2)$ and $I_d:F^p_{\alpha,w}\to L^p_{\alpha}(\mu)$.

\begin{lemma}\label{Delta}
Let $\alpha>0$, $p,r\geq1$, $w\in A^{\res}_{\infty}$, and let $\mu$ be a positive Borel measure on $\C$. For any $c=\{c_{\nu}\}\in l^p(\Z^2)$, define
$$\Delta(c):=\sum_{\nu\in\Z^2}c_{\nu}\frac{K_{\nu}}{\|K_{\nu}\|_{F^p_{\alpha,w}}}\chi_{Q_1(\nu)}.$$
Then $\Delta:l^p(\Z^2)\to L^p_{\alpha}(\mu)$ is $r$-summing if and only if $M_{\lambda}$ is $r$-summing on $l^p(\Z^2)$. Moreover,
$$\pi_r\big(\Delta:l^p(\Z^2)\to L^p_{\alpha}(\mu)\big)\asymp\pi_r\big(M_{\lambda}:l^p(\Z^2)\to l^p(\Z^2)\big).$$
\end{lemma}
\begin{proof}
For any $c=\{c_{\nu}\}\in l^p(\Z^2)$, by Lemma \ref{test},
\begin{align*}
\|\Delta(c)\|^p_{L^p_{\alpha}(\mu)}
&=\int_{\C}\left|\sum_{\nu\in\Z^2}c_{\nu}\frac{K_{\nu}(z)}{\|K_{\nu}\|_{F^p_{\alpha,w}}}\chi_{Q_1(\nu)}(z)\right|^p
  e^{-\frac{p\alpha}{2}|z|^2}d\mu(z)\\
&=\sum_{\nu\in\Z^2}\frac{|c_{\nu}|^p}{\|K_{\nu}\|^p_{F^p_{\alpha,w}}}\int_{Q_1(\nu)}
  |e^{\alpha\bar{\nu}z}|^pe^{-\frac{p\alpha}{2}|z|^2}d\mu(z)\\
&\asymp\sum_{\nu\in\Z^2}|c_{\nu}|^p\frac{1}{w(D(z,1))}\int_{Q_1(\nu)}
  e^{p\alpha\Re(\bar{\nu}z)-\frac{p\alpha}{2}|z|^2-\frac{p\alpha}{2}|\nu|^2}d\mu(z)\\
&\asymp\sum_{\nu\in\Z^2}|c_{\nu}|^p\frac{\mu(Q_1(\nu))}{w(Q_1(\nu))}=\|M_{\lambda}c\|^p_{l^p(\Z^2)}.
\end{align*}
Hence the desired result follows.
\end{proof}

\begin{lemma}\label{IdM}
Let $\alpha>0$, $p\geq1$, $w\in A^{\res}_{\infty}$, and let $\mu$ be a positive Borel measure on $\C$. If $I_d:F^p_{\alpha,w}\to L^p_{\alpha}(\mu)$ is $r$-summing for some $r\geq1$, then $M_{\lambda}$ is $r$-summing on $l^p(\Z^2)$ with
$$\pi_r\big(M_{\lambda}:l^p(\Z^2)\to l^p(\Z^2)\big)\lesssim
\pi_r\big(I_d:F^p_{\alpha,w}\to L^p_{\alpha}(\mu)\big).$$
\end{lemma}
\begin{proof}
Let $\{\delta_{\nu}\}_{\nu\in\Z^2}$ be a sequence of Rademacher functions on $[0,1]$. For every $\tau\in[0,1]$, define
\begin{align*}
\Phi_{\tau}:l^p(\Z^2)&\to F^p_{\alpha,w},\\
c=\{c_{\nu}\}&\mapsto\sum_{\nu\in\Z^2}c_{\nu}\delta_{\nu}(\tau)
  \frac{K_{\nu}}{\|K_{\nu}\|_{F^p_{\alpha,w}}}
\end{align*}
and
\begin{align*}
\Psi_{\tau}:L^p_{\alpha}(\mu)&\to L^p_{\alpha}(\mu),\\
f&\mapsto\sum_{\nu\in\Z^2}\delta_{\nu}(\tau)\chi_{Q_1(\nu)}f.
\end{align*}
Then by Lemma \ref{test}, $\Phi_{\tau}:l^p(\Z^2)\to F^p_{\alpha,w}$ is bounded, and $\|\Phi_{\tau}\|\lesssim1$ for every $\tau\in[0,1]$. On the other hand, for $f\in L^p_{\alpha}(\mu)$ and almost every $\tau\in[0,1]$,
\begin{align*}
\|\Psi_{\tau}f\|^p_{L^p_{\alpha}(\mu)}
&=\int_{\C}\left|\sum_{\nu\in\Z^2}\delta_{\nu}(\tau)\chi_{Q_1(\nu)}(z)f(z)\right|^p
  e^{-\frac{p\alpha}{2}|z|^2}d\mu(z)\\
&=\sum_{\nu\in\Z^2}\int_{Q_1(\nu)}|f(z)|^pe^{-\frac{p\alpha}{2}|z|^2}d\mu(z)=\|f\|^p_{L^p_{\alpha}(\mu)}.
\end{align*}
Therefore, $\Psi_{\tau}:L^p_{\alpha}(\mu)\to L^p_{\alpha}(\mu)$ is bounded, and $\|\Psi_{\tau}\|=1$.

We claim that for the embedding $I_d:F^p_{\alpha,w}\to L^p_{\alpha}(\mu)$,
\begin{equation}\label{delta-int}
\Delta=\int_0^1\Psi_{\tau}\circ I_d\circ\Phi_{\tau}d\tau,
\end{equation}
where $\Delta$ is the operator defined in Lemma \ref{Delta}. In fact, for any $c=\{c_{\nu}\}\in l^p(\Z^2)$,
\begin{align*}
\int_0^1\Psi_{\tau}\circ I_d\circ\Phi_{\tau}(c)d\tau
&=\int_0^1\Psi_{\tau}\left(\sum_{\nu\in\Z^2}c_{\nu}\delta_{\nu}(\tau)
  \frac{K_{\nu}}{\|K_{\nu}\|_{F^p_{\alpha,w}}}\right)d\tau\\
&=\int_0^1\sum_{\nu'\in\Z^2}\delta_{\nu'}(\tau)\chi_{Q_1(\nu')}\left(\sum_{\nu\in\Z^2}c_{\nu}\delta_{\nu}(\tau)
  \frac{K_{\nu}}{\|K_{\nu}\|_{F^p_{\alpha,w}}}\right)d\tau\\
&=\sum_{\nu,\nu'\in\Z^2}c_{\nu}\frac{K_{\nu}}{\|K_{\nu}\|_{F^p_{\alpha,w}}}\chi_{Q_1(\nu')}
  \int_0^1\delta_{\nu}(\tau)\delta_{\nu'}(\tau)d\tau\\
&=\sum_{\nu\in\Z^2}c_{\nu}\frac{K_{\nu}}{\|K_{\nu}\|_{F^p_{\alpha,w}}}\chi_{Q_1(\nu)}=\Delta(c),
\end{align*}
which gives \eqref{delta-int}. Therefore, if $I_d:F^p_{\alpha,w}\to L^p_{\alpha}(\mu)$ is $r$-summing, then $\Delta:l^p(\Z^2)\to L^p_{\alpha}(\mu)$ is $r$-summing. Consequently, Lemma \ref{Delta} implies that $M_{\lambda}$ is $r$-summing on $l^p(\Z^2)$. Moreover,
\begin{align*}
\pi_r\big(M_{\lambda}:l^p(\Z^2)\to l^p(\Z^2)\big)
&\asymp\pi_r\big(\Delta:l^p(\Z^2)\to L^p_{\alpha}(\mu)\big)\\
&\leq\int_0^1\pi_r\big(\Psi_{\tau}\circ I_d\circ\Phi_{\tau}:l^p(\Z^2)\to L^p_{\alpha}(\mu)\big)d\tau\\
&\lesssim\pi_r\big(I_d:F^p_{\alpha,w}\to L^p_{\alpha}(\mu)\big).
\end{align*}
The proof is complete.
\end{proof}

Combining Lemma \ref{IdM} with Lemmas \ref{mul} and \ref{dis}, we obtain the necessity part of Theorem \ref{main} for the case $p\geq2$.

\begin{corollary}
Let $\alpha>0$, $p\geq2$, $w\in A^{\res}_{\infty}$, and let $\mu$ be a positive Borel measure on $\C$. Suppose that $I_d:F^p_{\alpha,w}\to L^p_{\alpha}(\mu)$ is $r$-summing for some $r\geq1$.
\begin{enumerate}[(1)]
  \item If $r\leq p'$, then
  $$\left(\int_{\C}\widehat{\mu}_w(z)^{p'/p}dA(z)\right)^{1/p'}\lesssim
  \pi_r\big(I_d:F^p_{\alpha,w}\to L^p_{\alpha}(\mu)\big).$$
  \item If $p'\leq r\leq p$, then
  $$\left(\int_{\C}\widehat{\mu}_w(z)^{r/p}dA(z)\right)^{1/r}\lesssim
  \pi_r\big(I_d:F^p_{\alpha,w}\to L^p_{\alpha}(\mu)\big).$$
  \item If $r\geq p$, then
  $$\left(\int_{\C}\widehat{\mu}_w(z)dA(z)\right)^{1/p}\lesssim
  \pi_r\big(I_d:F^p_{\alpha,w}\to L^p_{\alpha}(\mu)\big).$$
\end{enumerate}
\end{corollary}

To establish the sufficiency part of Theorem \ref{main} in the case $p\geq2$, we need to determine the order boundedness of $I_d:F^p_{\alpha,w}\to L^p_{\alpha}(\mu)$, which has its own interest.

\begin{proposition}\label{order-bdd}
Let $\alpha>0$, $p\geq1$, $w\in A^{\res}_{\infty}$, and let $\mu$ be a positive Borel measure on $\C$. Then $I_d:F^p_{\alpha,w}\to L^p_{\alpha}(\mu)$ is order bounded if and only if $\widehat{\mu}_w\in L^1$.
\end{proposition}
\begin{proof}
It is clear that $I_d:F^p_{\alpha,w}\to L^p_{\alpha}(\mu)$ is order bounded if and only if
$$\int_{\C}\sup_{f\in B_{F^p_{\alpha,w}}}|f(z)|^pe^{-\frac{p\alpha}{2}|z|^2}d\mu(z)<\infty.$$
It follows from Lemmas \ref{pointwise} and \ref{test} that
$$\sup_{f\in B_{F^p_{\alpha,w}}}|f(z)|^p\asymp\frac{e^{\frac{p\alpha}{2}|z|^2}}{w(D(z,1))},\quad z\in\C.$$
Hence the order boundedness of $I_d:F^p_{\alpha,w}\to L^p_{\alpha}(\mu)$ is equivalent to
$$\int_{\C}\frac{d\mu(z)}{w(D(z,1))}<\infty.$$
Fubini's theorem together with Lemma \ref{esti} yields
$$\int_{\C}\frac{d\mu(z)}{w(D(z,1))}\asymp\int_{\C}\int_{D(z,1)}dA(u)\frac{d\mu(z)}{w(D(z,1))}
\asymp\int_{\C}\widehat{\mu}_w(u)dA(u),$$
which finishes the proof.
\end{proof}

As an application of Proposition \ref{order-bdd}, we can establish the sufficiency part of Theorem \ref{main} for the case $r\geq p\geq2$.

\begin{corollary}\label{suff1}
Let $\alpha>0$, $r\geq p\geq2$, $w\in A^{\res}_{\infty}$, and let $\mu$ be a positive Borel measure on $\C$. If $\widehat{\mu}_w\in L^1$, then $I_d:F^p_{\alpha,w}\to L^p_{\alpha}(\mu)$ is $r$-summing, and
$$\pi_r\big(I_d:F^p_{\alpha,w}\to L^p_{\alpha}(\mu)\big)\lesssim
\left(\int_{\C}\widehat{\mu}_w(z)dA(z)\right)^{1/p}.$$
\end{corollary}
\begin{proof}
If $\widehat{\mu}_w\in L^1$, then by Proposition \ref{order-bdd} and Lemma \ref{order}, $I_d:F^p_{\alpha,w}\to L^p_{\alpha}(\mu)$ is $p$-summing. Consequently, for any $r\geq p$,
\begin{align*}
\pi_r\big(I_d:F^p_{\alpha,w}\to L^p_{\alpha}(\mu)\big)
&\leq\pi_p\big(I_d:F^p_{\alpha,w}\to L^p_{\alpha}(\mu)\big)\\
&\leq\left\|\sup_{f\in B_{F^p_{\alpha,w}}}|f|\right\|_{L^p_{\alpha}(\mu)}
\lesssim\left(\int_{\C}\widehat{\mu}_w(z)dA(z)\right)^{1/p}.
\end{align*}
The proof is complete.
\end{proof}

To handle the reminder cases, we need some more notions. Given a sequence $\{X_{\nu}\}_{\nu\in\Z^2}$ of Banach spaces, the space $\oplus_{l^p(\Z^2)}X_{\nu}$ consists of sequences $\{x_{\nu}(\in X_{\nu})\}_{\nu\in\Z^2}$ with
$$\left\|\{x_{\nu}\}_{\nu\in\Z^2}\right\|_{\oplus_{l^p(\Z^2)}X_{\nu}}
:=\left(\sum_{\nu\in\Z^2}\|x_{\nu}\|^p_{X_{\nu}}\right)^{1/p}<\infty.$$
Fix $\alpha,p>0$ and a measurable subset $E\subset\C$. For a positive Borel measure $\mu$ on $\C$, let $L^p_{\alpha}(E,\mu)$ be the space of measurable functions $f$ on $E$ satisfying
$$\|f\|^p_{L^p_{\alpha}(E,\mu)}:=\int_E|f(z)|^pe^{-\frac{p\alpha}{2}|z|^2}d\mu(z)<\infty.$$
Let $F^p_{\alpha,w}(E)$ be the space of analytic functions $f$ on $E$ such that
\begin{enumerate}[i)]
  \item $f$ can be extended to entire functions; and
  \item $\|f\|^p_{F^p_{\alpha,w}(E)}:=\int_E|f(z)|^pe^{-\frac{p\alpha}{2}|z|^2}w(z)dA(z)<\infty.$
\end{enumerate}
Let $F^{\infty}_{\alpha}(E)$ be the space of analytic functions $f$ on $E$ such that
$$\|f\|_{F^{\infty}_{\alpha}(E)}:=\sup_{z\in E}|f(z)|e^{-\frac{\alpha}{2}|z|^2}<\infty.$$

It is clear that the discretization operator
\begin{align*}
\mathfrak{D}_1:L^p_{\alpha}(\mu)&\to\oplus_{l^p(\Z^2)}L^p_{\alpha}(Q_1(\nu),\mu),\\
f&\mapsto\{f|_{Q_1(\nu)}\}_{\nu\in\Z^2}
\end{align*}
is an isometric isomorphism, and the discretization operator
\begin{align*}
\mathfrak{D}_2:F^p_{\alpha,w}&\to\oplus_{l^p(\Z^2)}F^p_{\alpha,w}(D(\nu,2)),\\
f&\mapsto\{f|_{D(\nu,2)}\}_{\nu\in\Z^2}
\end{align*}
is bounded.

\begin{lemma}\label{Snu}
Let $\alpha>0$, $p\geq1$, $w\in A^{\res}_{\infty}$, and let $\mu$ be a positive Borel measure on $\C$. Suppose that $\mu(Q_1(\nu))<\infty$ for some $\nu\in\Z^2$. Then the operator $S_{\nu}$ defined by
\begin{align*}
S_{\nu}:F^p_{\alpha,w}(D(\nu,2))&\to L^p_{\alpha}(Q_1(\nu),\mu),\\
f&\mapsto f|_{Q_1(\nu)}
\end{align*}
is $1$-summing, and
$$\pi_1\big(S_{\nu}:F^p_{\alpha,w}(D(\nu,2))\to L^p_{\alpha}(Q_1(\nu),\mu)\big)
\lesssim\left(\frac{\mu(Q_1(\nu))}{w(Q_1(\nu))}\right)^{1/p}.$$
\end{lemma}
\begin{proof}
Define the operators $R_{\nu}$ and $T_{\nu}$ respectively by
\begin{align*}
R_{\nu}:F^p_{\alpha,w}(D(\nu,2))&\to F^{\infty}_{\alpha}(D(\nu,1)),\\
f&\mapsto w(Q_1(\nu))^{1/p}f|_{D(\nu,1)}
\end{align*}
and
\begin{align*}
T_{\nu}:F^1_{\alpha,w/\hat{w}(\nu)}(D(\nu,1))&\to L^p_{\alpha}(Q_1(\nu),\mu),\\
f&\mapsto w(Q_1(\nu))^{-1/p}f|_{Q_1(\nu)}.
\end{align*}
Then the following commutative diagram holds:
\begin{equation*}
\xymatrix{
  F^p_{\alpha,w}(D(\nu,2)) \ar[d]_{R_{\nu}} \ar[r]^{S_{\nu}}
                & L^p_{\alpha}(Q_1(\nu),\mu)   \\
  F^{\infty}_{\alpha}(D(\nu,1)) \ar[r]_{I_d}
                & F^1_{\alpha,w/\hat{w}(\nu)}(D(\nu,1)) \ar[u]^{T_{\nu}}.        }
\end{equation*}
We claim that $R_{\nu}:F^p_{\alpha,w}(D(\nu,2))\to F^{\infty}_{\alpha}(D(\nu,1))$ and $T_{\nu}:F^1_{\alpha,w/\hat{w}(\nu)}(D(\nu,1))\to L^p_{\alpha}(Q_1(\nu),\mu)$ are both bounded. In fact, for any $f\in F^p_{\alpha,w}(D(\nu,2))$, by Lemmas \ref{esti} and \ref{pointwise},
\begin{align*}
\|R_{\nu}f\|_{F^{\infty}_{\alpha}(D(\nu,1))}
&=\sup_{z\in D(\nu,1)}w(Q_1(\nu))^{1/p}|f(z)|e^{-\frac{\alpha}{2}|z|^2}\\
&\lesssim\sup_{z\in D(\nu,1)}\frac{w(Q_1(\nu))^{1/p}}{w(D(z,1))^{1/p}}
  \left(\int_{D(z,1)}|f(u)|^pe^{-\frac{p\alpha}{2}|u|^2}w(u)dA(u)\right)^{1/p}\\
&\lesssim\left(\int_{D(\nu,2)}|f(u)|^pe^{-\frac{p\alpha}{2}|u|^2}w(u)dA(u)\right)^{1/p}
  =\|f\|_{F^p_{\alpha,w}(D(\nu,2))}.
\end{align*}
Hence $R_{\nu}:F^p_{\alpha,w}(D(\nu,2))\to F^{\infty}_{\alpha}(D(\nu,1))$ is bounded, and $$\|R_{\nu}\|\lesssim1.$$
For any $f\in F^1_{\alpha,w/\hat{w}(\nu)}$, using Lemmas \ref{esti} and \ref{pointwise} again yields that
\begin{align*}
&\|T_{\nu}f\|^p_{L^p_{\alpha}(Q_1(\nu),\mu)}\\
&\ =\frac{1}{w(Q_1(\nu))}\int_{Q_1(\nu)}|f(z)|^pe^{-\frac{p\alpha}{2}|z|^2}d\mu(z)\\
&\ \lesssim\frac{1}{w(Q_1(\nu))}\int_{Q_1(\nu)}\frac{1}{w(D(z,1/4))^p}
  \left(\int_{D(z,1/4)}|f(u)|e^{-\frac{\alpha}{2}|u|^2}w(u)dA(u)\right)^pd\mu(z)\\
&\ \lesssim\frac{\mu(Q_1(\nu))}{w(Q_1(\nu))}\left(\int_{D(\nu,1)}|f(u)|
  e^{-\frac{\alpha}{2}|u|^2}\frac{w(u)}{w(Q_1(\nu))}dA(u)\right)^p\\
&\ =\frac{\mu(Q_1(\nu))}{w(Q_1(\nu))}\|f\|^p_{F^1_{\alpha,w/\hat{w}(\nu)}(D(\nu,1))},
\end{align*}
which implies that $T_{\nu}:F^1_{\alpha,w/\hat{w}(\nu)}(D(\nu,1))\to L^p_{\alpha}(Q_1(\nu),\mu)$ is bounded, and
$$\|T_{\nu}\|\lesssim\left(\frac{\mu(Q_1(\nu))}{w(Q_1(\nu))}\right)^{1/p}.$$
Since $w(D(\nu,1))\asymp\hat{w}(\nu)$ due to Lemma \ref{esti}, we know that $I_d:F^{\infty}_{\alpha}(D(\nu,1))\to F^1_{\alpha,w/\hat{w}(\nu)}(D(\nu,1))$ is $1$-summing. In fact, for any finite sequence $\{f_j\}_{1\leq j\leq n}\subset F^{\infty}_{\alpha}(D(\nu,1))$ with
$$\sup_{|a_j|\leq1}\left\|\sum_{j=1}^na_jf_j\right\|_{F^{\infty}_{\alpha}(D(\nu,1))}\leq1,$$
we have that
\begin{align*}
	&\sum_{j=1}^n\|f_j\|_{F^1_{\alpha,w/\hat{w}(\nu)}(D(\nu,1))}\\
	&\ =\sum_{j=1}^n\int_{D(\nu,1)}|f_j(z)|e^{-\frac{\alpha}{2}|z|^2}\frac{w(z)}{w(Q_1(\nu))}dA(z)\\
	&\ =\frac{1}{w(Q_1(\nu))}\int_{D(\nu,1)}\left(\sup_{|a_j|\leq1}
	\left|\sum_{j=1}^{n}a_jf_j(z)\right|\right)e^{-\frac{\alpha}{2}|z|^2}w(z)dA(z)\\
	&\ \leq\frac{w(D(\nu,1))}{w(Q_1(\nu))}\asymp1.
\end{align*}
Hence $I_d:F^{\infty}_{\alpha}(D(\nu,1))\to F^1_{\alpha,w/\hat{w}(\nu)}(D(\nu,1))$ is $1$-summing, and
$$\pi_1\big(I_d:F^{\infty}_{\alpha}(D(\nu,1))\to F^1_{\alpha,w/\hat{w}(\nu)}(D(\nu,1))\big)\lesssim1.$$
Therefore, $S_{\nu}=T_{\nu}\circ I_d\circ R_{\nu}:F^p_{\alpha,w}(D(\nu,2))\to L^p_{\alpha}(Q_1(\nu),\mu)$ is $1$-summing, and
$$\pi_1(S_{\nu})\lesssim\left(\frac{\mu(Q_1(\nu))}{w(Q_1(\nu))}\right)^{1/p}.$$
The proof is finished.
\end{proof}

The following theorem is the last piece of the proof of Theorem \ref{main}.

\begin{theorem}\label{suff2}
Let $\alpha>0$, $p\geq2$, $w\in A^{\res}_{\infty}$, and let $\mu$ be a positive Borel measure on $\C$.
\begin{enumerate}[(1)]
  \item If $p'\leq r\leq p$ and $\widehat{\mu}_w\in L^{r/p}$, then $I_d:F^p_{\alpha,w}\to L^p_{\alpha}(\mu)$ is $r$-summing, and
      $$\pi_r\big(I_d:F^p_{\alpha,w}\to L^p_{\alpha}(\mu)\big)
      \lesssim\left(\int_{\C}\widehat{\mu}_w(z)^{r/p}dA(z)\right)^{1/r}.$$
  \item If $\widehat{\mu}_w\in L^{p'/p}$, then $I_d:F^p_{\alpha,w}\to L^p_{\alpha}(\mu)$ is $1$-summing. Consequently, for any $1\leq r\leq p'$, $I_d:F^p_{\alpha,w}\to L^p_{\alpha}(\mu)$ is $r$-summing, and
      \begin{align*}
      \pi_r\big(I_d:F^p_{\alpha,w}\to L^p_{\alpha}(\mu)\big)
      &\leq\pi_1\big(I_d:F^p_{\alpha,w}\to L^p_{\alpha}(\mu)\big)\\
      &\lesssim\left(\int_{\C}\widehat{\mu}_w(z)^{p'/p}dA(z)\right)^{1/p'}.
      \end{align*}
\end{enumerate}
\end{theorem}
\begin{proof}
Let $S:\oplus_{l^p(\Z^2)}F^p_{\alpha,w}(D(\nu,2))\to\oplus_{l^p(\Z^2)}L^p_{\alpha}(Q_1(\nu),\mu)$ be the natural block diagonal mapping whose entries are the operators $S_{\nu}$ defined in Lemma \ref{Snu}. Then we have the following commutative diagram:
\begin{equation*}
\xymatrix{
  F^p_{\alpha,w} \ar[d]_{\mathfrak{D}_2} \ar[r]^{I_d}
                & L^p_{\alpha}(\mu) \ar[d]^{\mathfrak{D}_1}  \\
  \oplus_{l^p(\Z^2)}F^p_{\alpha,w}(D(\nu,2))  \ar[r]_{S}
                & \oplus_{l^p(\Z^2)}L^p_{\alpha}(Q_1(\nu),\mu).             }
\end{equation*}
Since $\mathfrak{D}_1$ is an isometric isomorphism, we have
$$I_d=\mathfrak{D}_1^{-1}\circ S\circ\mathfrak{D}_2.$$

(1) If $p'\leq r\leq p$ and $\widehat{\mu}_w\in L^{r/p}$, then by Lemma \ref{dis},
$$\sum_{\nu\in\Z^2}\left(\frac{\mu(Q_1(\nu))}{w(Q_1(\nu))}\right)^{r/p}<\infty,$$
which, in conjunction with Lemma \ref{Snu} and \cite[Corollary 4.8]{LR18}, implies that $S:\oplus_{l^p(\Z^2)}F^p_{\alpha,w}(D(\nu,2))\to\oplus_{l^p(\Z^2)}L^p_{\alpha}(Q_1(\nu),\mu)$ is $r$-summing. Moreover,
\begin{align*}
\pi_r(S)&\asymp\left(\sum_{\nu\in\Z^2}\pi_r(S_{\nu})^r\right)^{1/r}
\leq\left(\sum_{\nu\in\Z^2}\pi_1(S_{\nu})^r\right)^{1/r}\\
&\lesssim\left(\sum_{\nu\in\Z^2}\left(\frac{\mu(Q_1(\nu))}{w(Q_1(\nu))}\right)^{r/p}\right)^{1/r}
\asymp\left(\int_{\C}\widehat{\mu}_w(z)^{r/p}dA(z)\right)^{1/r}.
\end{align*}
Therefore, $I_d:F^p_{\alpha,w}\to L^p_{\alpha}(\mu)$ is $r$-summing, and
$$\pi_r\big(I_d:F^p_{\alpha,w}\to L^p_{\alpha}(\mu)\big)\lesssim\pi_r(S)
\lesssim\left(\int_{\C}\widehat{\mu}_w(z)^{r/p}dA(z)\right)^{1/r}.$$

(2) If $\widehat{\mu}_w\in L^{p'/p}$, then by Lemmas \ref{dis}, \ref{Snu} and \cite[Corollary 9.2]{HJLL}, we obtain that $S:\oplus_{l^p(\Z^2)}F^p_{\alpha,w}(D(\nu,2))\to\oplus_{l^p(\Z^2)}L^p_{\alpha}(Q_1(\nu),\mu)$ is $1$-summing. Moreover,
\begin{align*}
\pi_1(S)\lesssim\left(\sum_{\nu\in\Z^2}\pi_1(S_{\nu})^{p'}\right)^{1/p'}
&\lesssim\left(\sum_{\nu\in\Z^2}\left(\frac{\mu(Q_1(\nu))}{w(Q_1(\nu))}\right)^{p'/p}\right)^{1/p'}\\
&\asymp\left(\int_{\C}\widehat{\mu}_w(z)^{p'/p}dA(z)\right)^{1/p'}.
\end{align*}
Therefore, $I_d:F^p_{\alpha,w}\to L^p_{\alpha}(\mu)$ is $1$-summing, and
$$\pi_1\big(I_d:F^p_{\alpha,w}\to L^p_{\alpha}(\mu)\big)
\lesssim\left(\int_{\C}\widehat{\mu}_w(z)^{p'/p}dA(z)\right)^{1/p'},$$
which completes the proof.
\end{proof}

Combining Corollary \ref{suff1} with Theorem \ref{suff2} finishes the sufficiency part of Theorem \ref{main} for the case $p\geq2$.

As a direct consequence of Theorem \ref{main}, we can obtain the following monotonicity for embeddings on standard Fock spaces (see also \cite[Theorem 8.4]{LR18} and \cite[Proposition 11.1]{HJLL}).

\begin{corollary}
Let $\alpha>0$, $1<q\leq p<\infty$, $r\geq1$, and let $\mu$ be a positive Borel measure on $\C$. If $I_d:F^p_{\alpha}\to L^p_{\alpha}(\mu)$ is $r$-summing, then $I_d:F^q_{\beta}\to L^q_{\beta}(\mu)$ is $r$-summing for any $\beta>0$.
\end{corollary}

\section{Applications}\label{app}

In this section, we apply Theorem \ref{main} to establish some results for differentiation and integration operators, Volterra-type operators and composition operators.

\subsection{Differentiation and integration operators}

Given a nonnegative integer $k$, we define the differentiation operator $D^{(k)}$ by
$$D^{(k)}f:=f^{(k)},\quad f\in\h(\C).$$
If $k$ is a negative integer, then the integration operator $D^{(k)}$ is defined for entire functions $f$ by
$$D^{(k)}f:=D^{(k+1)}\circ D^{(-1)}f,$$
where
$$D^{(-1)}f(z):=\int_0^zf(\zeta)d\zeta,\quad z\in\C.$$

In \cite{CFP}, using the Littlewood--Paley formula of $F^p_{\alpha,w}$, Cascante, F\`{a}brega and Pel\'{a}ez characterized the boundedness of the differentiation and integration operators $D^{(k)}:F^p_{\alpha,w}\to L^p_{\alpha}(\mu)$ for all $\alpha,p>0$, $k\in\mathbb{Z}$ and $w\in A^{\res}_{\infty}$. We now use Theorem \ref{main} to characterize the $r$-summability of $D^{(k)}:F^p_{\alpha,w}\to L^p_{\alpha}(\mu)$. For a weight $w$ on $\C$ and $\gamma\in\mathbb{R}$, we write
$$w_{\gamma}(z):=\frac{w(z)}{(1+|z|)^{\gamma}},\quad z\in\C.$$

\begin{theorem}
Let $\alpha>0$, $k\in\mathbb{Z}$, and let $\mu$ be a positive Borel measure on $\C$ such that
$$\sup_{l\in\{0\}\cup\mathbb{N}}\int_{\C}|z|^{lp}e^{-\frac{p\alpha}{2}|z|^2}d\mu(z)<\infty.$$
\begin{enumerate}[(1)]
  \item Let $1<p<2$ and $w\in A^{\res}_p$. Then for any $r\geq1$, $D^{(k)}:F^p_{\alpha,w}\to L^p_{\alpha}(\mu)$ is $r$-summing if and only if
      $$\int_{\C}\left(\frac{\mu(D(z,1))}{w_{kp}(D(z,1))}\right)^{2/p}dA(z)<\infty.$$
  \item Let $p\geq2$ and $w\in A^{\res}_{\infty}$. Then
  \begin{enumerate}[(i)]
    \item for $1\leq r\leq p'$, $D^{(k)}:F^p_{\alpha,w}\to L^p_{\alpha}(\mu)$ is $r$-summing if and only if
        $$\int_{\C}\left(\frac{\mu(D(z,1))}{w_{kp}(D(z,1))}\right)^{p'/p}dA(z)<\infty.$$
    \item for $p'\leq r\leq p$, $D^{(k)}:F^p_{\alpha,w}\to L^p_{\alpha}(\mu)$ is $r$-summing if and only if
        $$\int_{\C}\left(\frac{\mu(D(z,1))}{w_{kp}(D(z,1))}\right)^{r/p}dA(z)<\infty.$$
    \item for $r\geq p$, $D^{(k)}:F^p_{\alpha,w}\to L^p_{\alpha}(\mu)$ is $r$-summing if and only if
        $$\int_{\C}\frac{\mu(D(z,1))}{w_{kp}(D(z,1))}dA(z)<\infty.$$
  \end{enumerate}
\end{enumerate}
\end{theorem}
\begin{proof}
It is sufficient to prove that $D^{(k)}:F^p_{\alpha,w}\to L^p_{\alpha}(\mu)$ is $r$-summing if and only if $I_d:F^p_{\alpha,w_{kp}}\to L^p_{\alpha}(\mu)$ is $r$-summing. Assume first that $I_d:F^p_{\alpha,w_{kp}}\to L^p_{\alpha}(\mu)$ is $r$-summing. Choose any finite sequence $\{f_j\}_{1\leq j\leq n}\subset F^p_{\alpha,w}$ such that
$$\sup_{\{c_j\}\in B_{l^{r'}}}\left\|\sum_{j=1}^nc_jf_j\right\|_{F^p_{\alpha,w}}\leq1.$$
We want to show
$$\left(\sum_{j=1}^n\|D^{(k)}f_j\|^r_{L^p_{\alpha}(\mu)}\right)^{1/r}\leq C$$
for some absolute constant $C>0$. Note that by \cite[Theorem 1.1]{CFP}, $D^{(k)}f_j\in F^p_{\alpha,w_{kp}}$, and for any $\{c_j\}\in B_{l^{r'}}$,
$$\left\|\sum_{j=1}^nc_jD^{(k)}f_j\right\|_{F^p_{\alpha,w_{kp}}}
=\left\|D^{(k)}\left(\sum_{j=1}^nc_jf_j\right)\right\|_{F^p_{\alpha,w_{kp}}}
\lesssim\left\|\sum_{j=1}^nc_jf_j\right\|_{F^p_{\alpha,w}}\leq1.$$
Hence the $r$-summability of $I_d:F^p_{\alpha,w_{kp}}\to L^p_{\alpha}(\mu)$ implies
$$\left(\sum_{j=1}^{n}\left\|D^{(k)}f_j\right\|^r_{L^p_{\alpha}(\mu)}\right)^{1/r}\lesssim
\pi_r\big(I_d:F^p_{\alpha,w_{kp}}\to L^p_{\alpha}(\mu)\big),$$
which is exactly what we need.

Conversely, assume that $D^{(k)}:F^p_{\alpha,w}\to L^p_{\alpha}(\mu)$ is $r$-summing. Choose $\{f_j\}_{1\leq j\leq n}\subset F^p_{\alpha,w_{kp}}$ such that
$$\sup_{\{c_j\}\in B_{l^{r'}}}\left\|\sum_{j=1}^nc_jf_j\right\|_{F^p_{\alpha,w_{kp}}}\leq1.$$
It follows again from \cite[Theorem 1.1]{CFP} that $D^{(-k)}f_j\in F^p_{\alpha,w}$, and for $\{c_j\}\in B_{l^{r'}}$,
$$\left\|\sum_{j=1}^nc_jD^{(-k)}f_j\right\|_{F^p_{\alpha,w}}\lesssim1.$$
Hence the $r$-summability of $D^{(k)}:F^p_{\alpha,w}\to L^p_{\alpha}(\mu)$ gives
$$\left(\sum_{j=1}^n\left\|D^{(k)}\left(D^{(-k)}f_j\right)\right\|^r_{L^p_{\alpha}(\mu)}\right)^{1/r}\lesssim
\pi_r\big(D^{(k)}:F^p_{\alpha,w}\to L^p_{\alpha}(\mu)\big).$$
In the case $k\geq0$, it is clear that
\begin{align*}
\left(\sum_{j=1}^n\|f_j\|^r_{L^p_{\alpha}(\mu)}\right)^{1/r}&=
\left(\sum_{j=1}^n\left\|D^{(k)}\left(D^{(-k)}f_j\right)\right\|^r_{L^p_{\alpha}(\mu)}\right)^{1/r}\\
&\lesssim\pi_r\big(D^{(k)}:F^p_{\alpha,w}\to L^p_{\alpha}(\mu)\big).
\end{align*}
Hence $I_d:F^p_{\alpha,w_{kp}}\to L^p_{\alpha}(\mu)$ is $r$-summing. In the case $k<0$, write
$$C_{\mu,\alpha,k}:=\max_{l=0,1,\dots,-k-1}\left(\int_{\C}|z|^{lp}e^{-\frac{p\alpha}{2}|z|^2}d\mu(z)\right)^{1/p},$$
and let $\delta^{(l)}_0$ be the point evaluation of the $l$-th derivative at $0$, that is,
$$\delta^{(l)}_0f:=f^{(l)}(0),\quad f\in\h(\C).$$
Then, noting that for any $f\in\h(\C)$,
$$f(z)=D^{(k)}\left(D^{(-k)}f\right)(z)+\sum_{l=0}^{-k-1}\frac{f^{(l)}(0)}{l!}z^l,$$
we obtain that
\begin{align*}
&\left(\sum_{j=1}^n\|f_j\|^r_{L^p_{\alpha}(\mu)}\right)^{1/r}\\
&\ =\left(\sum_{j=1}^n\left\|D^{(k)}\left(D^{(-k)}f_j\right)(z)+
  \sum_{l=0}^{-k-1}\frac{f_j^{(l)}(0)}{l!}z^l\right\|^r_{L^p_{\alpha}(\mu)}\right)^{1/r}\\
&\ \leq\left(\sum_{j=1}^n\left\|D^{(k)}\left(D^{(-k)}f_j\right)\right\|^r_{L^p_{\alpha}(\mu)}\right)^{1/r}+
  \sum_{l=0}^{-k-1}\left(\sum_{j=1}^n\frac{|f_j^{(l)}(0)|^r}{(l!)^r}\|z^l\|^r_{L^p_{\alpha}(\mu)}\right)^{1/r}\\
&\ \lesssim\pi_r\big(D^{(k)}:F^p_{\alpha,w}\to L^p_{\alpha}(\mu)\big)
  +C_{\mu,\alpha,k}\sum_{l=0}^{-k-1}\frac{1}{l!}\left(\sum_{j=1}^n|\delta^{(l)}_0(f_j)|^r\right)^{1/r}\\
&\ \leq\pi_r\big(D^{(k)}:F^p_{\alpha,w}\to L^p_{\alpha}(\mu)\big)
  +C_{\mu,\alpha,k}\sum_{l=0}^{-k-1}\frac{1}{l!}\|\delta^{(l)}_0\|_{(F^p_{\alpha,w_{kp}})^*}.
\end{align*}
Therefore, $I_d:F^p_{\alpha,w_{kp}}\to L^p_{\alpha}(\mu)$ is $r$-summing, and
\begin{align*}
\pi_r\big(I_d:F^p_{\alpha,w_{kp}}&\to L^p_{\alpha}(\mu)\big)\\
&\lesssim
\pi_r\big(D^{(k)}:F^p_{\alpha,w}\to L^p_{\alpha}(\mu)\big)
  +C_{\mu,\alpha,k}\sum_{l=0}^{-k-1}\frac{1}{l!}\|\delta^{(l)}_0\|_{(F^p_{\alpha,w_{kp}})^*}.
\end{align*}
The proof is complete.
\end{proof}

\subsection{Volterra-type operators}

Given $g\in\h(\C)$, the Volterra-type operator $J_g$ is defined for entire functions $f$ by
$$J_gf(z):=\int_0^zf(\zeta)g'(\zeta)d\zeta,\quad z\in\C.$$
The boundedness and compactness of the operators $J_g$ acting on the Fock spaces $F^p_{\alpha}$ were investigated in \cite{Co12,Hu13}. We here consider the $r$-summability of $J_g$. Using Theorem \ref{main} and \cite[Theorem 1.1]{CFP}, one can easily establish the integrability conditions for the $r$-summability of $J_g$ acting on $F^p_{\alpha,w}$. However, in the standard Fock space setting, we can obtain a more explicit result.

\begin{theorem}\label{Vol}
Let $\alpha>0$, $p>1$, $r\geq1$, and let $g\in\h(\C)$.
\begin{enumerate}[(1)]
  \item If $p\leq2$, or $p>2$ and $r\leq2$, then $J_g$ is $r$-summing on $F^p_{\alpha}$ if and only if $g$ is constant.
  \item If $p>2$ and $r>2$, then $J_g$ is $r$-summing on $F^p_{\alpha}$ if and only if $g(z)=az+b$ for some $a,b\in\C$.
\end{enumerate}
\end{theorem}
\begin{proof}
Suppose that $p\leq2$ and $J_g$ is $r$-summing on $F^p_{\alpha}$ for some $r\geq1$. Note that for any $f\in F^p_{\alpha}$, \cite[Proposition 1]{Co12} yields
$$\|J_gf\|^p_{F^p_{\alpha}}
\asymp\int_{\C}|f(z)|^pe^{-\frac{p\alpha}{2}|z|^2}\frac{|g'(z)|^p}{(1+|z|)^p}dA(z)
=\|f\|^p_{L^p_{\alpha}(\mu_{g,p})},$$
where $d\mu_{g,p}(z)=\frac{|g'(z)|^p}{(1+|z|)^p}dA(z)$. Hence the embedding $I_d:F^p_{\alpha}\to L^p_{\alpha}(\mu_{g,p})$ is $r$-summing. Then by Theorem \ref{main},
\begin{equation}\label{finite}
\int_{\C}\mu_{g,p}(D(z,1))^{2/p}dA(z)
=\int_{\C}\left(\int_{D(z,1)}\frac{|g'(u)|^p}{(1+|u|)^p}dA(u)\right)^{2/p}dA(z)<\infty,
\end{equation}
which together with the sub-harmonic property of $|g'|$ implies
$$\frac{|g'(\xi)|^2}{(1+|\xi|)^2}\lesssim
\int_{D(\xi,1)}\left(\int_{D(z,1)}\frac{|g'(u)|^p}{(1+|u|)^p}dA(u)\right)^{2/p}dA(z)\to0$$
as $|\xi|\to\infty$. Therefore, $g(z)=az+b$ for some $a,b\in\C$. If $g$ is not constant, i.e. $a\neq0$, then
$$\int_{\C}\left(\int_{D(z,1)}\frac{|g'(u)|^p}{(1+|u|)^p}dA(u)\right)^{2/p}dA(z)
\asymp|a|^2\int_{\C}\frac{dA(z)}{(1+|z|)^2}=\infty,$$
which contradicts \eqref{finite} and finishes the proof for the case $p\leq2$. The other cases are similar and so are omitted.
\end{proof}

Using Theorem \ref{Vol}, we can investigate the boundedness of $J_g$ acting from weak to strong vector-valued Fock spaces. It was proved in \cite[Corollary 4.6]{CW21} that for any complex infinite-dimensional Banach space $X$, $J_g:wF^2_{\alpha}(X)\to F^2_{\alpha}(X)$ is bounded if and only if $g$ is constant, and if $p>2$, then $J_g:wF^p_{\alpha}(X)\to F^p_{\alpha}(X)$ is bounded if and only if $g(z)=az+b$ for some $a,b\in\C$. In particular, it was pointed out in \cite[Remark 4.8]{CW21} that $F^p_{\alpha}(X)\subsetneq wF^p_{\alpha}(X)$ for any $p\geq2$, $\alpha>0$ and any complex infinite-dimensional Banach space $X$. We here generalize these results to the case $1<p<2$.

\begin{theorem}\label{Vol-sum}
Let $1<p<2$, $\alpha>0$, $g\in\h(\C)$, and let $X$ be a complex infinite-dimensional Banach space. Then the following assertions hold.
\begin{enumerate}[(1)]
  \item $J_g:wF^p_{\alpha}(X)\to F^p_{\alpha}(X)$ is bounded if and only if $g$ is constant.
  \item $F^p_{\alpha}(X)\subsetneq wF^p_{\alpha}(X)$.
\end{enumerate}
\end{theorem}
\begin{proof}
(1) Suppose that $J_g:wF^p_{\alpha}(X)\to F^p_{\alpha}(X)$ is bounded. In view of Theorem \ref{Vol}, it is sufficient to prove that $J_g$ is $2$-summing on $F^p_{\alpha}$. Note that $(F^p_{\alpha})^*=F^{p'}_{\alpha}$ under the pairing $\langle\cdot,\cdot\rangle_{\alpha}$. Fix a positive integer $n$ and suppose that $\{f_j\}_{1\leq j\leq n}\subset F^p_{\alpha}$ with
$$\sup_{h\in B_{F^{p'}_{\alpha}}}\left(\sum_{j=1}^n|\langle f_j,h\rangle_{\alpha}|^2\right)\leq1.$$
We may use the Dvoretzky theorem (see \cite[Theorem 19.1]{DJT}) to find a linear embedding $T_n:l^n_2\to X$  such that for any $\{c_j\}_{1\leq j\leq n}\in l^n_2$,
\begin{equation}\label{Dvo}
\frac{1}{2}\left(\sum_{j=1}^n|c_j|^2\right)^{1/2}\leq\left\|\sum_{j=1}^nc_jT_n(e_j)\right\|_X
\leq\left(\sum_{j=1}^n|c_j|^2\right)^{1/2},
\end{equation}
where $(e_1,\dots,e_n)$ is some fixed orthonormal basis of $l^n_2$. Define
$$F_n(z)=\sum_{j=1}^nf_j(z)T_n(e_j),\quad z\in\C.$$
Then $F_n\in wF^p_{\alpha}(X)$ and $\|F_n\|_{wF^p_{\alpha}(X)}\leq1$ (see page 10 of \cite{Bl20}). The boundedness of $J_g:wF^p_{\alpha}(X)\to F^p_{\alpha}(X)$ together with \cite[Theorem 4.1]{CW21} and \eqref{Dvo} gives that
\begin{align*}
\|J_g\|^p_{wF^p_{\alpha}(X)\to F^p_{\alpha}(X)}
&\geq\|J_gF_n\|^p_{F^p_{\alpha}(X)}\\
&\asymp\int_{\C}\left\|\sum_{j=1}^nf_j(z)T_n(e_j)\right\|^p_X|g'(z)|^pe^{-\frac{p\alpha}{2}|z|^2}
  \frac{dA(z)}{(1+|z|)^p}\\
&\geq\frac{1}{2^p}\int_{\C}\left(\sum_{j=1}^n|f_j(z)|^2\right)^{p/2}|g'(z)|^pe^{-\frac{p\alpha}{2}|z|^2}
  \frac{dA(z)}{(1+|z|)^p}.
\end{align*}
Noting that $p/2\leq1$, combining Minkowski's inequality with \cite[Theorem 4.1]{CW21} yields that
\begin{align*} 
\|J_g\|^p_{wF^p_{\alpha}(X)\to F^p_{\alpha}(X)}
&\gtrsim\left(\sum_{j=1}^n\left(\int_{\C}|f_j(z)|^p|g'(z)|^pe^{-\frac{p\alpha}{2}|z|^2}
  \frac{dA(z)}{(1+|z|)^p}\right)^{2/p}\right)^{p/2}\\
&\asymp\left(\sum_{j=1}^n\|J_gf_j\|^2_{F^p_{\alpha}}\right)^{p/2},
\end{align*}
which implies that $J_g$ is $2$-summing on $F^p_{\alpha}$ and finishes the proof.

(2) Using (1), the desired result follows from the same argument as in \cite[Remark 4.8]{CW21}.
\end{proof}

\subsection{Composition operators}

The $r$-summing composition operators acting on the Fock spaces $F^p_{\alpha}$ are characterized as follows.

\begin{theorem}\label{com}
Let $\alpha>0$ and $\varphi\in\h(\C)$. Then for every $p>1$ and every $r\geq1$, $C_{\varphi}$ is $r$-summing on $F^p_{\alpha}$ if and only if $\varphi(z)=az+b$ for some $a,b\in\C$ with $|a|<1$.
\end{theorem}
\begin{proof}
Fix $p>1$ and $r\geq1$. Let $\mu_{\varphi}$ be the pull-back measure defined by
$$\mu_{\varphi}(E):=\int_{\varphi^{-1}(E)}e^{-\frac{p\alpha}{2}(|z|^2-|\varphi(z)|^2)}dA(z)$$
for every Borel subset $E\subset\C$. Then for any $f\in F^p_{\alpha}$,
$$\|C_{\varphi}f\|^p_{F^p_{\alpha}}=\int_{\C}|f(\varphi(z))|^pe^{-\frac{p\alpha}{2}|z|^2}dA(z)
=\int_{\C}|f(z)|^pe^{-\frac{p\alpha}{2}|z|^2}d\mu_{\varphi}(z).$$
Therefore, $C_{\varphi}$ is $r$-summing on $F^p_{\alpha}$ if and only if the embedding $I_d:F^p_{\alpha}\to L^p_{\alpha}(\mu_{\varphi})$ is $r$-summing.

If $C_{\varphi}$ is $r$-summing on $F^p_{\alpha}$, then it is naturally bounded. Consequently, $\varphi(z)=az+b$ for some $a,b\in\C$ with $|a|<1$, or $|a|=1$ and $b=0$ (see \cite[Theorem 1]{CMS} or \cite[p. 89]{Zh}). To proceed, we split the proof into three cases.

{\bf Case 1.} $|a|=1$ and $b=0$. Then it is clear that for any $z\in\C$, $\mu_{\varphi}(D(z,1))=\pi$, so for any $q>0$, $\mu_{\varphi}(D(\cdot,1))\notin L^q$. Hence by Theorem \ref{main}, $C_{\varphi}$ is not $r$-summing on $F^p_{\alpha}$.

{\bf Case 2.} $0<|a|<1$. Then for any $z\in\C$,
\begin{align*}
\mu_{\varphi}(D(z,1))&=\int_{\varphi^{-1}(D(z,1))}e^{-\frac{p\alpha}{2}(|u|^2-|\varphi(u)|^2)}dA(u)\\
&=\int_{D(\frac{z-b}{a},\frac{1}{|a|})}e^{-\frac{p\alpha}{2}(|u|^2-|au+b|^2)}dA(u)\\
&\lesssim\int_{D(\frac{z-b}{a},\frac{1}{|a|})}e^{-\frac{p\alpha}{4}(1-|a|^2)|u|^2}dA(u)\\
&\lesssim\int_{D(\frac{z-b}{a},\frac{1}{|a|})}e^{-\frac{p\alpha}{8}(1-|a|^2)\left|\frac{z-b}{a}\right|^2}
  e^{-\frac{p\alpha}{4}(1-|a|^2)\left|u-\frac{z-b}{a}\right|^2}dA(u)\\
&\asymp e^{-\frac{p\alpha(1-|a|^2)}{8|a|^2}|z-b|^2},
\end{align*}
which implies $\mu_{\varphi}(D(\cdot,1))\in L^q$ for any $q>0$. Hence by Theorem \ref{main}, $C_{\varphi}$ is $r$-summing on $F^p_{\alpha}$.

{\bf Case 3.} $a=0$. Then for any $z\in\C$,
$$\mu_{\varphi}(D(z,1))
=\int_{\C}e^{-\frac{p\alpha}{2}(|u|^2-|b|^2)}dA(u)\cdot\chi_{D(b,1)}(z).$$
Hence $C_{\varphi}$ is $r$-summing on $F^p_{\alpha}$ by Theorem \ref{main}.
\end{proof}

Similarly as in the case of Volterra-type operators, we can use Theorem \ref{com} to obtain the boundedness of composition operators on vector-valued Fock spaces. Note that it was proved in \cite[Corollary 4.7]{CW21} that for any $p\geq2$, $\alpha>0$ and any complex infinite-dimensional Banach space $X$, $C_{\varphi}:wF^p_{\alpha}(X)\to F^p_{\alpha}(X)$ is bounded if and only if $\varphi(z)=az+b$ for some $a,b\in\C$ with $|a|<1$. We here show that this still holds for $1<p<2$.

\begin{theorem}
Let $1<p<2$, $\alpha>0$, $\varphi\in\h(\C)$, and let $X$ be a complex infinite-dimensional Banach space. Then $C_{\varphi}:wF^p_{\alpha}(X)\to F^p_{\alpha}(X)$ is bounded if and only if $\varphi(z)=az+b$ for some $a,b\in\C$ with $|a|<1$.
\end{theorem}
\begin{proof}
The necessity part can be proven by the same method as in the proof of Theorem \ref{Vol-sum} (see also \cite{Bl20} and \cite[Theorem 11.3]{HJLL}). We omit it here. The sufficiency part is actually contained in the proof of \cite[Theorem 4.4]{CW21}. In fact, if $\varphi(z)=az+b$ for some $a,b\in\C$ with $|a|<1$, then by \cite[Theorem 2.7]{Zh}, for any $f\in wF^p_{\alpha}(X)$,
\begin{align*}
\|C_{\varphi}f\|^p_{F^p_{\alpha}(X)}&=\int_{\C}\|f(\varphi(z))\|^p_{X}e^{-\frac{p\alpha}{2}|z|^2}dA(z)\\
&\leq\|f\|^p_{wF^p_{\alpha}(X)}\int_{\C}e^{-\frac{p\alpha}{2}(|z|^2-|az+b|^2)}dA(z).
\end{align*}
Hence $C_{\varphi}:wF^p_{\alpha}(X)\to F^p_{\alpha}(X)$ is bounded.
\end{proof}

\medskip


\noindent{\bf Acknowledgment.} The authors would like to thank the referee for valuable suggestions.

\end{document}